\documentclass[a4paper,twoside,12pt]{article}
\usepackage[english]{babel}

\usepackage{amsmath,amsfonts,amssymb,amsthm}
\newtheorem{theorem}{Theorem}[section]
\newtheorem{lemma}[theorem]{Lemma}

\theoremstyle{definition}
\newtheorem{definition}[theorem]{Definition}

\theoremstyle{remark}
\newtheorem{remark}[theorem]{Remark} 

 \title{{\bf On Spectral Properties of some Class of
Non-selfadjoint Operators}}

\author{Maksim \,V.~Kukushkin   \\ \\
  \small  \textit{International Committee  Continental,}\\\textit{\small\textit{Russia, Geleznovodsk, kukushkinmv@rambler.ru}} }
\date{}

\begin{document}

\maketitle

{ \it
\begin{flushright}
 Blessed memory of  Isai I. Mikaelian  is  devoted
\end{flushright}
}

\begin{abstract}
In this paper  we    explore a certain  class of non-selfadjoint operators acting    in  a complex  separable Hilbert space. We consider a  perturbation of a non-selfadjoint operator by an operator  that is also  non-selfadjoint. Our consideration is based  on    known spectral  properties of the    real component of a  non-selfadjoint compact  operator. Using a technic of the sesquilinear form theory we establish the   compactness property of the  resolvent,   obtain the  asymptotic equivalence between   the
real component of the resolvent and the resolvent of the real component for some class  of   non-selfadjoint operators.
We obtain   a classification of non-selfadjoint operators in accordance with   belonging   their    resolvent  to   the  Schatten-von Neumann  class and formulate  a sufficient condition of completeness of the root vectors system. Finally we obtain  an   asymptotic formula for  eigenvalues of the considered class of non-selfadjoint operators.
\end{abstract} \maketitle

\section{Introduction}
It is remarkable that initially the perturbation theory of selfadjoint operators was born in  the  works of M. Keldysh
\cite{firstab_lit:1keldysh 1951}-\cite{firstab_lit:3keldysh 1971} and had been  motivated by  the   works of such famous scientists as T. Carleman \cite{firstab_lit:1Carleman} and Ya.  Tamarkin  \cite{firstab_lit:1Tamarkin}.   Over time   many papers were  published within the    framework of this theory, for instance    F. Browder \cite{firstab_lit:1Browder},  M. Livshits \cite{firstab_lit:1Livshits}, B. Mukminov \cite{firstab_lit:1Mukminov}, I. Glazman \cite{firstab_lit:1Glazman}, M. Krein \cite{firstab_lit:1Krein}, B. Lidsky \cite{firstab_lit:1Lidskii},  A. Marcus \cite{firstab_lit:1Markus},\cite{firstab_lit:2Markus}, V. Matsaev \cite{firstab_lit:1Matsaev}-\cite{firstab_lit:3Matsaev}, S. Agmon \cite{firstab_lit:2Agmon}, V. Katznelson \cite{firstab_lit:1Katsnel'son}.
Nowadays there  exists   a  huge amount of theoretical results formulated  in the work of A. Shkalikov    \cite{firstab_lit:Shkalikov A.}. However for applying  these results for a concrete operator  we must have  a representation   of one   by  the sum of   the   main part  (in the other words  a  so-called   non-perturbing operator) and the   operator-perturbation.
  It is essential   that the main part must be an operator of a special type either a  selfadjoint or a  normal operator. If we consider a case  where  in the representation   the  main part is  neither selfadjoint nor normal and we cannot approach the  required representation in an  obvious  way, then it is possible to   use  another technique  based on    properties  of the real component of the initial operator. This is a subject to consider  in the second  section. In the third  section   we
demonstrate the  significance   of the obtained abstract results  and consider  concrete operators. Note that the  relevance   of such consideration  is based on the following.
 The eigenvalue problem is still relevant for  the  second order fractional  differential operators.  Many papers were  devoted to this question, for instance the papers
\cite{firstab_lit:1Nakhushev1977}, \cite{firstab_lit:1Aleroev1984}-\cite{firstab_lit:3Aleroev2000}.   The singular number   problem for the resolvent of  the second order differential operator
  with the  Riemann-Liouville  fractional derivative in the  final  term  is considered in the paper \cite{firstab_lit:1Aleroev1984}. It is proved
that the resolvent   belongs to the Hilbert-Schmidt   class.  The problem
of    root functions  system completeness  is researched in the paper \cite{firstab_lit:Aleroev1989}, also a similar problem is considered
in the paper \cite{firstab_lit:3Aleroev2000}. We would like to study spectral properties of some class of  non-selfadjoint operators in   the abstract case.  Via obtained results  we   research a  multidimensional case of the  second order fractional  differential operator   which can be reduced to the
  cases   considered in the papers listed above. For this purpose we deal with the extension of the
Kipriyanov fractional differential operator considered in detail in the papers \cite{firstab_lit:kipriyanov1960}-\cite{firstab_lit:2.2kipriyanov1960}.

\section{Preliminaries}

 Let    $ C,C_{i} ,\;i\in \mathbb{N}_{0}$ be positive real constants. We   assume     that  the    values  of $C$   can be different in   various formulas  but the values of $C_{i} ,\;i\in \mathbb{N}_{0}$ are  certain.  Everywhere further  we consider   linear    densely defined operators acting in a separable complex  Hilbert space $\mathfrak{H}$. Denote by $ \mathcal{B} (\mathfrak{H})$    the set of linear bounded operators acting in   $\mathfrak{H}.$
 Denote by    $   \mathrm{D}(L),\,   \mathrm{R}(L),\,  \mathrm{N}(L)   $    the   domain of definition, the range, and the inverse image of zero of the operator $L$ accordingly. Let    $\mathrm{P}(L)$ be  a resolvent set of the operator $L.$
 Denote by    $ R_{L}(\zeta),\,\zeta\in \mathrm{P}(L),\,[R_{L} :=R_{L}(0)]$ the  resolvent of the operator $L.$ Let  $\lambda_{i}(L),\,i\in \mathbb{N} $ denote the eigenvalues of the operator $L.$
 Suppose $L$ is  a compact operator and  $|L|:=(L^{\ast}L)^{1/2},\,r(|L|):={\rm dim}\, \mathrm{R} (|L|);$ then   the eigenvalues of the operator $|L|$ are called the   {\it singular  numbers} ({\it s-numbers}) of the operator $L$ and are denoted by $s_{i}(L),\,i=1,\,2,...\,,r(|L|).$ If $r(|L|)<\infty,$ then we put by definition     $s_{i}=0,\,i=r(|L|)+1,2,...\,.$
 According  to the terminology of the monograph   \cite{firstab_lit:1Gohberg1965}  the dimension  of the  root vectors subspace  corresponding  to a certain  eigenvalue $\lambda_{k}$  is called     the    {\it algebraic multiplicity} of the eigenvalue   $\lambda_{k}.$
Let  $\nu(L)$ denote  the sum of all algebraic multiplicities of the operator $L.$ Denote by  $\mathfrak{S}_{p}(\mathfrak{H}),\, 0< p<\infty $         the Schatten-von Neumann    class and   let   $\mathfrak{S}_{\infty}(\mathfrak{H})$ denote  the set of compact operators. By definition, put
$$
\mathfrak{S}_{p}(\mathfrak{H}):=\left\{ L: \mathfrak{H}\rightarrow \mathfrak{H},  \sum\limits_{i=1}^{\infty}s^{p}_{i}(L)<\infty,\;0< p<\infty \right\}.
$$
    Suppose  $L$ is  an   operator that has    a  compact resolvent and
$s_{n}(R_{L})\leq   C \,n^{-\mu},\,n\in \mathbb{N},\,0\leq\mu< \infty;$ then
 we
 denote by  $\mu(L) $   order of the     operator $L$ in accordance with  the definition given in the paper  \cite{firstab_lit:Shkalikov A.}.
 Denote by  $L_{\mathfrak{R}}:= \left(L+L^{*}\right)/2,\,L_{\mathfrak{I}}:= \left(L-L^{*}\right)/2 i$
   the  real  and the   imaginary component    of the operator $L$ accordingly  and let
    $   \tilde{L}   $ denote    the  closure of  the operator $L.$
In accordance with  the terminology of the monograph  \cite{firstab_lit:kato1980} the set $\Theta(L):=\{z\in \mathbb{C}: z=(Lf,f)_{\mathfrak{H}},\,f\in \mathrm{D}(L),\,\|f\|_{\mathfrak{H}}=1\}$ is called   the   {\it numerical range}  of  the  operator $L.$
 We use the definition of  the   sectorial property given in \cite[p.280]{firstab_lit:kato1980}.  An  operator $L$ is called   a {\it sectorial } operator  if its  numerical range   belongs to a  closed
sector     $\mathfrak{L}_{\gamma}(\theta):=\{\zeta:\,|\arg(\zeta-\gamma)|\leq\theta<\pi/2\} ,$ where      $\gamma$ is the  vertex   and  $ \theta$ is the  semi-angle of the sector   $\mathfrak{L}_{\gamma}(\theta).$
We  shall  say that the  operator $L$ has a positive sector       if $\mathrm{Im}\,\gamma=0,\,\gamma>0.$
According to the  terminology  of the monograph   \cite{firstab_lit:kato1980} an operator $L$ is called    {\it strictly  accretive}   if the following relation  holds  $\mathrm{Re}(Lf,f)_{\mathfrak{H}}\geq C\|f\|^{2}_{\mathfrak{H}},\,f\in \mathrm{D}(L).$ In accordance with  the definition  \cite[p.279]{firstab_lit:kato1980}   an  operator $L$ is called    {\it m-accretive}     if the next relation  holds $(A+\zeta)^{-1}\in \mathcal{B}(\mathfrak{H}),\,\|(A+\zeta)^{-1}\| \leq   (\mathrm{Re}\zeta)^{-1},\,\mathrm{Re}\zeta>0. $
An operator $L$ is called    {\it m-sectorial}   if $L$ is   sectorial    and $L+ \beta$ is m-accretive   for some constant $\beta.$   An operator $L$ is called     {\it symmetric}     if one is densely defined and the next equality  holds $(Lf,g)_{\mathfrak{H}}=(f,Lg)_{\mathfrak{H}},\,f,g\in  \mathrm{D} (L).$
A symmetric operator is called     {\it positive}        if     the   values of its  quadratic form  are nonnegative.
Denote by $\mathfrak{H}_{L},\,\|\cdot\|_{L}$ the   energetic space generated by the operator $L$  and the  norm on this space respectively   (see  \cite{firstab_lit:Eb. Zeidler},\cite{firstab_lit:mihlin1970}).
In accordance with  the denotation of the paper   \cite{firstab_lit:kato1980}   we consider a      sesquilinear form   $ \mathfrak{t}  [\cdot,\cdot]$
defined on a linear manifold  of the Hilbert space $\mathfrak{H}$ (further we use the term {\it form}).  Denote by $  \mathfrak{t} [\cdot ]$ the  quadratic form corresponding to the sesquilinear form $\mathfrak{t}  [\cdot,\cdot].$
Let   $\mathfrak{Re} \,  \mathfrak{t}=(\mathfrak{t}+\mathfrak{t}^{\ast})/2,\,\mathfrak{Im}\,\mathfrak{t}   =(\mathfrak{t}-\mathfrak{t}^{\ast})/2i$
   be the   real  and    imaginary component     of the   form $ \mathfrak{t}$ respectively, where $\mathfrak{t}^{\ast}[u,v]=\mathfrak{t} \overline{[v,u]},\;\mathrm{D}(\mathfrak{t} ^{\ast})=\mathrm{D}(\mathfrak{t}).$ According to these definitions, we have $\mathfrak{Re} \,
 \mathfrak{t}[\cdot]=\mathrm{Re}\,\mathfrak{t}[\cdot],\,\mathfrak{Im} \, \mathfrak{t}[\cdot]=\mathrm{Im}\,\mathfrak{t}[\cdot].$ Denote by $\tilde{\mathfrak{t}}$ the  closure   of the   form $\mathfrak{t}\,.$  The range of a quadratic form    $\mathfrak{t}[f],\,f\in \mathrm{D}(\mathfrak{t}),\,\|f\|_{\mathfrak{H}}=1$ is called the  {\it range} of the sesquilinear form  $\mathfrak{t} $ and is denoted by $\Theta(\mathfrak{t}).$
 A  form $\mathfrak{t}$ is called    {\it sectorial}    if  its    range  belongs to   a sector  having  the vertex $\gamma$  situated at the real axis and the semi-angle $0\leq\theta<\pi/2.$   Suppose   $\mathfrak{l}$ is a closed sectorial form; then  a linear  manifold  $\mathrm{D}'\subset\mathrm{D} (\mathfrak{l})$   is
called the  core of $\mathfrak{l}$ if the restriction   of $\mathfrak{l}$ to   $\mathrm{D}'$ has the   closure
$\mathfrak{l}.$    Due to Theorem 2.7 \cite[p.323]{firstab_lit:kato1980}  there exist unique    m-sectorial operators  $L_{\mathfrak{l}},L_{\mathfrak{Re}\, \mathfrak{l}} $  associated  with   the  closed sectorial   forms $\mathfrak{l},\mathfrak{Re}\, \mathfrak{l}$   respectively.   The operator  $L_{\mathfrak{Re}\, \mathfrak{l}} $ is called the {\it real part} of the operator $L_{\mathfrak{l}}$ and is denoted by  $Re\, L_{\mathfrak{l}}.$ Suppose  $L$ is a sectorial densely defined operator and $\mathfrak{k}[u,v]:=(Lu,v)_{\mathfrak{H}},\,\mathrm{D}(\mathfrak{k})=\mathrm{D}(L);$  then
 due to   Theorem 1.27 \cite[p.318]{firstab_lit:kato1980}   the   form $\mathfrak{k}$ is   closable, due to
   Theorem 2.7 \cite[p.323]{firstab_lit:kato1980} there exists   a unique m-sectorial operator   $T_{\tilde{\mathfrak{k}}}$   associated  with  the form $\tilde{\mathfrak{k}}.$  In accordance with the  definition \cite[p.325]{firstab_lit:kato1980} the    operator $T_{\tilde{\mathfrak{k}}}$ is called the   {\it Friedrichs extension} of the operator $L.$

  Further,    if it is not     stated otherwise   we  use   the  notations   of  the  monographs   \cite{firstab_lit:1Gohberg1965},  \cite{firstab_lit:kato1980}, \cite{firstab_lit:samko1987}.
  Consider a pair of  complex  separable Hilbert spaces $\mathfrak{H},\mathfrak{H}_{+}$ such that
\begin{equation}\label{2.1}
\mathfrak{H}_{+}\hookrightarrow\hookrightarrow\mathfrak{ H} .
\end{equation}
This denotation  implies  that $\mathfrak{H}_{+}$ is dense in $\mathfrak{H}$   and we have a bounded embedding provided by the inequality
\begin{equation}\label{2.2}
\|f\|_{\mathfrak{H}}\leq \|f\|_{\mathfrak{H}_{+}},\;f\in \mathfrak{H}_{+},
\end{equation}
moreover   any  bounded  set in the space   $\mathfrak{H}_{+} $ is a compact set in   the space $\mathfrak{H} .$
We  consider   non-selfadjoint operators which can be represented by a sum $W=T+A.$ The operators $T$ and $A$  are called     a {\it main part} and an {\it  operator-perturbation}   respectively,     both   these operators  act   in $\mathfrak{H}.$     We   assume  that:  there  exists a linear manifold  $\mathfrak{M}\subset \mathfrak{H}_{+} $    that   is  dense in $ \mathfrak{H}_{+},$   the operators $T,A$ and their adjoint operators are   defined on $\mathfrak{M}.$   Further,  we may assume that      $\mathrm{D}(W)=\mathfrak{M}.$ This  gives us the opportunity to  prove   that  $\mathrm{D}(W )\subset \mathrm{D}(W^{\ast}) .$
  Suppose the  operator $W^{+}$ is the   restriction of   $W^{\ast}$ to $ \mathrm{D}(W );$   then the operator $W^{+}$ is called a {\it formal adjoint} operator with respect to $W.$
 Denote  by $\tilde{W}^{+} $ the closure of the operator $W^{+}.$
 Further, we assume  that the following conditions are fulfilled
\begin{equation}\label{2.3}
\mathrm{i})\,\mathrm{Re}(Tf,f)_{\mathfrak{H}}\geq C_{0}\|f\|^{2}_{\mathfrak{H}_{+}}\!,\;\mathrm{ii})\,
\left| (Tf,g)_{\mathfrak{H}}\right|\leq C_{1}\|f\|_{\mathfrak{H}_{+}}\|g\|_{\mathfrak{H}_{+}},
$$
$$
 \mathrm{iii})\,\mathrm{Re}(Af,f)_{\mathfrak{H}}\geq C_{2} \|f\|^{2}_{\mathfrak{H}}, \; \mathrm{iv})\,|(Af,g)_{\mathfrak{H}}|\leq C_{3}\|f\|_{\mathfrak{H}_{+}}\|g\|_{\mathfrak{H} },\,f,g\in  \mathfrak{M}.
\end{equation}
 Due to these conditions  it is easy to prove  that the operators $W,W_{\mathfrak{R}}$ are closeable (see   Theorem 3.4 \cite[p.268]{firstab_lit:kato1980}).
Denote by $\tilde{W}_{\mathfrak{R}}$ the  closure  of the operator $W_{\mathfrak{R}}.$
  To make  some formulas readable we also use the following form of notation
 $V:=  \left(R_{\tilde{W}}\right)_{\mathfrak{R}},\, H:=\tilde{W}_{\mathfrak{R}}.$

\section{Main results}

In this section  we   formulate abstract theorems that  are generalizations of some particular  results obtained by the author. First  we generalize Theorem 4.2 \cite{firstab_lit:1kukushkin2018} establishing the sectorial property of the second order fractional differential operator.
\begin{lemma}\label{L3.1}
 The operators $\tilde{W},\,\tilde{W}^{+}$ have a positive sector.
\end{lemma}
\begin{proof}  Due to    inequalities   \eqref{2.2},\eqref{2.3}   we   conclude that the  operator $W$ is strictly accretive, i.e.
\begin{equation}\label{3.1}
\mathrm{Re}(Wf,f)_{\mathfrak{H} }\geq C_{0} \|f\|^{2}_{\mathfrak{H} } ,\;f\in \mathrm{D}(W).
\end{equation}
 Let us prove that the operator $\tilde{W}$ is canonical sectorial. Combining   \eqref{2.3}  (ii) and  \eqref{2.3} (iii), we get
\begin{equation}\label{3.2}
\mathrm{Re}(Wf,f)_{\mathfrak{H}}= \mathrm{Re}(Tf,f)_{\mathfrak{H}}+\mathrm{Re}(Af,f)_{\mathfrak{H}}\geq  C_{0}\|f\|_{\mathfrak{H}_{+}}+C_{2}\|f\|_{\mathfrak{H}},\,f\in \mathrm{D}(W) .
\end{equation}
  Obviously    we can extend the previous  inequality  to
\begin{equation}\label{3.3}
\mathrm{Re}(\tilde{W}f,f)_{\mathfrak{H}} \geq  C_{0}\|f\|_{\mathfrak{H}_{+}}+C_{2}\|f\|_{\mathfrak{H}} ,\,f\in \mathrm{D}(\tilde{W}).
\end{equation}
By virtue of \eqref{3.3}, we obtain  $\mathrm{D}(\tilde{W})\subset\mathfrak{H}_{+}. $
Note that we have the estimate
\begin{equation*}
|\mathrm{Im}(Wf,f)_{\mathfrak{H}}|\leq \left|\mathrm{Im}(Tf,f)_{\mathfrak{H}}\right|+\left|\mathrm{Im}(Af,f)_{\mathfrak{H}}\right|=
I_{1}+I_{2},\,f\in \mathrm{D}(W) .
\end{equation*}
Using        inequality \eqref{2.3} (ii),  the    Jung  inequality,   we get
$$
I_{1}=\left| (Tv,u)_{\mathfrak{H}}- (Tu,v)_{\mathfrak{H}}\right|\leq  \left| (Tv,u)_{\mathfrak{H}}\right|+\left| (Tu,v)_{\mathfrak{H}}\right|\leq   2 C_{1}\|u\|_{\mathfrak{H}_{+}}\|v\|_{\mathfrak{H}_{+}}\leq  C_{1} \|f\|^{2}_{\mathfrak{H}_{+}},
$$
where $f=u+i\, v.$ Consider $I_{2}.$
Applying  the  Cauchy Schwartz inequality and inequality \eqref{2.3} (iv),  we obtain for arbitrary positive $\varepsilon$
$$
 \left| (Av,u)_{\mathfrak{H}} \right|\leq C_{3}\|v\|_{\mathfrak{H}_{+}}\|u\|_{\mathfrak{H}} \leq
\frac{C_{3}}{2}\left\{\frac{1}{\varepsilon}\|u\|^{2}_{\mathfrak{H}}+\varepsilon\|v\|^{2}_{\mathfrak{H}_{+}}\right\}\,;
$$
$$
 \left| (Au,v)_{\mathfrak{H}}\right|\leq
\frac{C_{3}}{2}\left\{\frac{1}{\varepsilon}\|v\|^{2}_{\mathfrak{H}}+\varepsilon\|u\|^{2}_{\mathfrak{H}_{+}}\right\}.
$$
 Hence
$$
I_{2}=\left| (Av,u)_{\mathfrak{H}} -(Au,v)_{\mathfrak{H}}\right|\leq \left| (Av,u)_{\mathfrak{H}}| +|(Au,v)_{\mathfrak{H}}\right|\leq \frac{C_{3}}{2}\left\{\frac{1}{\varepsilon}\|f\|^{2}_{\mathfrak{H}}+\varepsilon\|f\|^{2}_{\mathfrak{H}_{+}}\right\}.
$$
  Finally,   we have the following estimate
$$
|\mathrm{Im}(Wf,f)_{\mathfrak{H}}|  \leq  \frac{C_{3}}{2}\, \varepsilon^{-1} \|f\|^{2}_{\mathfrak{H}}+\left(\frac{C_{3}}{2}\,\varepsilon+C_{1}\right)\|f\|^{2}_{\mathfrak{H}_{+}} ,\,f\in \mathrm{D}(W).
 $$
Thus,  we   conclude that      the next inequality holds for   arbitrary $k>0$
\begin{equation*}
{\rm Re}( Wf, f    )_{\mathfrak{H}}-k \left|{\rm Im} ( Wf, f    )_{\mathfrak{H}}\right|\geq
$$
$$
\geq\left[C_{0}-k\left(\frac{C_{3}}{2}\,\varepsilon+C_{1}\right)\right] \|f\|^{2}_{\mathfrak{H}_{+}} +\left(C_{2}-k\,\frac{C_{3}}{2}\, \varepsilon^{-1}\right)\|f\|^{2}_{\mathfrak{H}},\,f\in \mathrm{D}(W) .
\end{equation*}
  Using the continuity property of the inner product, we can extend the previous  inequality  to  the set  $\mathrm{D}(\tilde{W}).$
It follows easily that
\begin{equation}\label{3.4}
\left|{\rm Im} \left( [\tilde{W}-\gamma(\varepsilon)]f, f    \right)_{\mathfrak{H}}\right|   \leq \frac{1}{k(\varepsilon)}
{\rm Re}\left( [\tilde{W}-\gamma(\varepsilon)]f, f    \right)_{\mathfrak{H}}\! ,\,f\in \mathrm{D}(\tilde{W}) ,
$$
$$
k(\varepsilon)= C_{0}\left(\frac{ C_{3} }{2} \,\varepsilon+ C_{1}\right)^{-1},
\;\gamma(\varepsilon)=C_{2}-  k(\varepsilon) \, \frac{  C_{3} }{2}\,  \varepsilon^{-1} .
\end{equation}
The previous  inequality implies that the numerical range   of the operator $ \tilde{W} $ belongs to the sector $\mathfrak{L}_{\gamma}(\theta)$ with the vertex situated   at the point  $\gamma$ and the semi-angle $\theta=\arctan(1/k).$
Solving system of equations  \eqref{3.4} relative to $\varepsilon$ we obtain  the positive root $\xi$ corresponding to  the value $\gamma=0$ and the following description  for the coordinates of  the   sector vertex $\gamma$
\begin{equation*}
\gamma:=\left\{ \begin{aligned}
  \gamma<0,\;\varepsilon\in(0,\xi)  ,\\
  \gamma\geq 0,\; \varepsilon\in [\xi,\infty)   \\
\end{aligned}
 \right., \xi=\sqrt{\left(\frac{C_{1}}{ C_{3}}\right)^{2}+ \frac{C_{0}}{C_{2}}   }-\frac{C_{1}}{ C_{3}}.
\end{equation*}
 It follows that the operator $\tilde{W}$ has a positive sector.   The proof corresponding to    the operator  $ \tilde{W}^{+} $ follows   from the  reasoning  given  above if we note  that   $W ^{+}$    is formal adjoint with respect to   $W.$
  \end{proof}
  \begin{lemma}\label{L3.2}
  The operators $\tilde{W},\tilde{W}^{+}$   are m-accretive, their  resolvent sets contain  the half-plane $\{\zeta:\,\zeta\in \mathbb{C},\,\mathrm{Re}\,\zeta<C_{0}\}.$
\end{lemma}
\begin{proof}    Due to Lemma \ref{L3.1}  we know that the operator $\tilde{W}$ has a positive sector,  i.e. the numerical range of    $\tilde{W}$   belongs to the  sector
 $\mathfrak{L}_{\gamma}(\theta),\,\gamma>0.$  In consequence of Theorem 3.2   \cite[p.268]{firstab_lit:kato1980},
 we have   $\forall\zeta\in \mathbb{C}\setminus\mathfrak{L}_{\gamma}(\theta),$ the set  $\mathrm{R}(\tilde{W}-\zeta)$ is a closed  space,   and the next relation     holds
\begin{equation*}
  {\rm def}(\tilde{W}-\zeta)=\eta,\; \eta={\rm const} .
 \end{equation*}
Due to  Theorem 3.2   \cite[p.268]{firstab_lit:kato1980} the inverse operator $(\tilde{W}+\zeta)^{-1}$ is defined  on the subspace $\mathrm{R}(\tilde{W}+\zeta),\,{\rm Re}\zeta>0.$
In accordance with the definition of m-accretive operator  given in the monograph   \cite[p.279]{firstab_lit:kato1980}   we need to show  that
\begin{equation*}
{\rm def}(\tilde{W}+\zeta)=0,\;\|(\tilde{W}+\zeta)^{-1}\|\leq ({\rm Re}\zeta)^{-1},\,{\rm Re}\zeta>0.
\end{equation*}
  For this purpose assume that
 $\zeta_{0} \in\mathbb{C}\setminus \mathfrak{L}_{\gamma}(\theta),\;{\rm Re}\zeta_{0}  <0.$
Using  \eqref{3.1}, we get
 \begin{equation}\label{3.5}
  {\rm Re}  \left( f,[\tilde{W}-\zeta_{0} ]f  \right)_{\!\!\mathfrak{H}}\!\!\geq  (C_{0}- {\rm Re} \zeta_{0} ) \|f\|^{2}_{\mathfrak{H}},\;f\in \mathrm{D}(\tilde{W}).
 \end{equation}
Since the  operator $\tilde{W}-\zeta_{0}$    has the closed range    $\mathrm{R} (\tilde{W}-\zeta_{0}),$ it follows that
\begin{equation*}
 \mathfrak{H}=\mathrm{R} (\tilde{W}-\zeta_{0})\oplus \mathrm{R} (\tilde{W}-\zeta_{0})^{\perp} .
 \end{equation*}
Note that the  intersection of the  sets  $  \mathfrak{M}$ and $\mathrm{R} (\tilde{W}-\zeta_{0})^{\perp}$ is   zero. If we assume   otherwise, then applying  inequality  \eqref{3.5}   for any element
 $u\in \mathfrak{M}\cap \mathrm{R}  (\tilde{W}-\zeta_{0})^{\perp}$    we get
 \begin{equation*}
(C_{0}-{\rm Re}\zeta_{0}) \|u\|^{2}_{\mathfrak{H}} \leq {\rm Re} \left( u,[\tilde{W}-\zeta_{0}]u  \right)_{\mathfrak{H}}=0,
 \end{equation*}
hence $u=0.$ Thus  the intersection of the sets  $  \mathfrak{M}$ and $\mathrm{R} (\tilde{W}-\zeta_{0})^{\perp}$ is   zero. It implies that
$$
\left(g,v\right)_{\mathfrak{H}}=0,\;\forall g\in  \mathrm{R}  (\tilde{W}-\zeta_{0})^{\perp},\;\forall v\in \mathfrak{M}.
$$
Since $ \mathfrak{M}$   is    a dense set in $\mathfrak{H}_{+},$  then   taking into account  \eqref{2.2}, we obtain that      $ \mathfrak{M}$  is  a  dense  set  in $\mathfrak{H}.$ Hence    $\mathrm{R}  (\tilde{W}-\zeta_{0})^{\perp}=0,\,{\rm def} (\tilde{W}-\zeta_{0}) =0.$ Combining this fact with Theorem 3.2   \cite[p.268]{firstab_lit:kato1980}, we get   ${\rm def} (\tilde{W}-\zeta )=0,\;\zeta\in \mathbb{C}\setminus\mathfrak{L}_{\gamma}(\theta).$  It is clear that
${\rm def} (\tilde{W}+\zeta )=0,\,\forall\zeta,\,{\rm Re}\zeta>0.$  Let us prove that $\|(\tilde{W}+\zeta)^{-1}\|\leq ({\rm Re}\zeta)^{-1},\,\forall\zeta,\,{\rm Re}\zeta>0.$
  We must   notice that
 \begin{equation*}
(C_{0}+{\rm Re}\zeta ) \|f\|^{2}_{\mathfrak{H}} \leq {\rm Re} \left( f,[\tilde{W}+\zeta ]f  \right)_{\mathfrak{H}}\leq \|f\|_{\mathfrak{H}}\|(\tilde{W}+\zeta )f\|_{\mathfrak{H}},\;f\in \mathrm{D}(\tilde{W}),\;
{\rm Re}\zeta>0 .
 \end{equation*}
By virtue of the fact   ${\rm def} (\tilde{W}+\zeta )=0,\;\forall\zeta,\,{\rm Re}\zeta>0$  we   know   that the  resolvent   is   defined. Therefore
 \begin{equation*}
\|(\tilde{W}+\zeta )^{-1}f\|_{\mathfrak{H}}     \leq(C_{0}+{\rm Re}\,\zeta ) ^{-1} \|f\|_{\mathfrak{H}}\leq ( {\rm Re}\,\zeta ) ^{-1} \|f\|_{\mathfrak{H}},\;f\in \mathfrak{H}.
 \end{equation*}
 It implies that
$$
\|(\tilde{W}+\zeta )^{-1} \| \leq( {\rm Re}\,\zeta ) ^{-1},\;\forall\zeta,\,{\rm Re}\zeta>0.
$$
If we combine inequality \eqref{3.3}  with     Theorem 3.2 \cite[p.268]{firstab_lit:kato1980}, we get    $\mathrm{P}(\tilde{W})\supset\{\zeta:\,\zeta\in \mathbb{C},\,\mathrm{Re}\,\zeta<C_{0}\}.$
The  proof corresponding to the operator  $\tilde{W}^{+}$    is absolutely analogous.
\end{proof}

 \begin{lemma}\label{L3.3}
 The operator $\tilde{W}_{\mathfrak{R}}$ is strictly accretive,  m-accretive,  selfadjoint.
\end{lemma}
\begin{proof}
  It is  obvious  that $  W _{\mathfrak{R}}  $ is a symmetric operator. Due to the continuity  property of the inner product we can conclude that
   $\tilde{W}_{\mathfrak{R}}  $ is   symmetric too.  Hence  $\Theta(\tilde{W}_{\mathfrak{R}})\subset \mathbb{R}.$ By virtue of  \eqref{3.2}, we have
 \begin{equation*}
(W_{\mathfrak{R}}f,f)_{\mathfrak{H}}\geq C_{0}\|f\|^{2}_{\mathfrak{H}_{+}},\,f\in \mathrm{D}(W).
\end{equation*}
Using inequality \eqref{2.2} and  the  continuity property of  the   inner product, we obtain
 \begin{equation}\label{3.6}
(\tilde{W}_{\mathfrak{R}}f,f)_{\mathfrak{H}}\geq C_{0}\|f\|^{2}_{\mathfrak{H}_{+}}\geq C_{0}\|f\|^{2}_{\mathfrak{H}},\,f\in \mathrm{D}(\tilde{W}_{\mathfrak{R}}).
\end{equation}
It implies that $\tilde{W}_{\mathfrak{R}}$ is strictly accretive.
  In the same way as in the proof of Lemma \ref{L3.2}   we   come  to conclusion that
$\tilde{W}_{\mathfrak{R}}$ is m-accretive. Moreover we obtain the   relation ${\rm def}(\tilde{W}_{\mathfrak{R}}-\zeta )=0,\,  {\rm Im }\zeta\neq 0 .$
Hence by virtue  of Theorem 3.16  \cite[p.271]{firstab_lit:kato1980} the operator $\tilde{W}_{\mathfrak{R}}$ is selfadjoint.
\end{proof}
 \begin{theorem}\label{T3.4}
 The  operators $\tilde{W}_{\mathfrak{R}},\tilde{W},\tilde{W}^{+}$   have    compact resolvents.
\end{theorem}
\begin{proof}
First note that due to Lemma \ref{L3.3} the operator $\tilde{W}_{\mathfrak{R}}$  is selfadjoint.  Using \eqref{3.6},  we obtain    the estimates
$$
\|f\|_{ H }\geq  \sqrt{C_{0} }\|f\|_{\mathfrak{H}_{+}}\geq \sqrt{C_{0} } \|f\|_{\mathfrak{H} }, \;f\in \mathfrak{H}_{H},
$$
where $H:=\tilde{W}_{\mathfrak{R}}.$
Since  $\mathfrak{H}_{+}\hookrightarrow\hookrightarrow\mathfrak{H} ,$ then  we   conclude that each   set   bounded     with respect to the energetic norm generated by the operator $\tilde{W}_{\mathfrak{R}}$ is  compact with respect to the norm $\|\cdot\|_{\mathfrak{H}}.$     Hence in accordance  with Theorem   \cite[p.216]{firstab_lit:mihlin1970} we conclude that   $\tilde{W}_{\mathfrak{R}}$ has a discrete spectrum.  Note that  in consequence of  Theorem 5 \cite[p.222]{firstab_lit:mihlin1970} we have that    a selfadjoint strictly accretive    operator with  discrete spectrum has a compact inverse operator. Thus using Lemma \ref{L3.3}, Theorem 6.29 \cite[p.187]{firstab_lit:kato1980} we obtain  that $\tilde{W}_{\mathfrak{R}}$ has a compact resolvent.

 Further, we need the  technique of the   sesquilinear form theory  stated    in \cite{firstab_lit:kato1980}.
Consider the sesquilinear forms
$$ \mathfrak{t} [f,g]=(\tilde{W}f,g)_{\mathfrak{H}},\,f,g\in \mathrm{D}(\tilde{W}),\;  \mathfrak{h} [f,g]=(\tilde{W}_{\mathfrak{R}}f,g)_{\mathfrak{H}},\,f,g\in \mathrm{D}(\tilde{W}_{\mathfrak{R}}).
$$
 Recall  that due to   inequality \eqref{3.3} we came to the  conclusion that $\mathrm{D}(\tilde{W})\subset \mathfrak{H}_{+}.$ In the same way  we can deduce  that  $\mathrm{D}(\tilde{W}_{\mathfrak{R}})\subset \mathfrak{H}_{+}.$
By virtue   of   Lemma  \ref{L3.1},  Lemma \ref{L3.3}, it is easy to prove that the   sesquilinear forms $\mathfrak{t},\mathfrak{h}$  are sectorial. Applying   Theorem 1.27 \cite[p.318]{firstab_lit:kato1980}  we get that these  forms  are closable. Now note that   $\mathfrak{Re}\, \tilde{\mathfrak{t}}$   is a sum of two closed sectorial forms. Hence   in consequence of Theorem 1.31 \cite[p.319]{firstab_lit:kato1980},  we have that   $\mathfrak{Re}\,\tilde{\mathfrak{t}}$   is a closed form. Let us show that $\mathfrak{Re}\,\tilde{\mathfrak{t}}=\tilde{\mathfrak{h}}.$    First  note that this equality is true   on the  elements of the linear manifold $\mathfrak{M} \subset\mathfrak{H}_{+}.$   This  fact can be obtained  from the following obvious  relations
$$\tilde{\mathfrak{t}}  [f,g]=(  W   f,g)_{\mathfrak{H}},
\;  \overline{\tilde{\mathfrak{t}}[g,f]} = ( W ^{+}\! f,g)_{\mathfrak{H}},\,f,g\in \mathfrak{M}.
$$
On the other hand
$$\tilde{\mathfrak{h}}[f,g]=(\tilde{W}_{\mathfrak{R}}f,g)_{\mathfrak{H}}=(W_{\mathfrak{R}}f,g)_{\mathfrak{H}},\,f,g\in \mathfrak{M}.
$$
Hence
\begin{equation}\label{3.7}
\mathfrak{Re}\,\tilde{\mathfrak{t}}[f,g]=\tilde{\mathfrak{h}}[f,g],\;f,g\in \mathfrak{M}.
\end{equation}
 Using  \eqref{2.3}, we get
\begin{equation}\label{3.8}
 C_{0}\|f\|^{2}_{\mathfrak{H}_{+}}\leq \mathrm{Re}\,\tilde{\mathfrak{t}}[f] \leq C_{4} \|f\|^{2}_{\mathfrak{H}_{+}},\,C_{0} \|f\|^{2}_{\mathfrak{H}_{+}}\leq\tilde{\mathfrak{h}}[f]\leq  C_{4}\|f\|^{2}_{\mathfrak{H}_{+}},\,f\in \mathfrak{M},
\end{equation}
where $C_{4}=C_{1}+C_{3}.$  Since $ \mathfrak{Re} \,\tilde{\mathfrak{t}}[f]=\mathrm{ Re } \,\tilde{\mathfrak{t}}[f],\,f\in \mathfrak{M},$   the sesquilinear forms  $\mathfrak{Re} \,\tilde{\mathfrak{t}},\tilde{\mathfrak{h}}$ are closed forms, then using \eqref{3.8} it is easy to prove that $\mathrm{D}(\mathfrak{Re}\,\tilde{\mathfrak{t}})=\mathrm{D}(\tilde{\mathfrak{h}})=\mathfrak{H}_{+}.$ Using   estimates \eqref{3.8},   it is not hard to prove  that   $\mathfrak{M}$ is a core of the forms $\mathfrak{Re} \,\tilde{\mathfrak{t}},\tilde{\mathfrak{h}}.$ Hence using \eqref{3.7}, we obtain
$\mathfrak{Re}\,\tilde{\mathfrak{t}}[f] =\tilde{\mathfrak{h}}[f],\,f\in \mathfrak{H}_{+}.$ In accordance with    the   polarization principle (see (1.1) \cite[p.309]{firstab_lit:kato1980}),  we have     $\mathfrak{Re}\,\tilde{\mathfrak{t}}=\tilde{\mathfrak{h}}.$
Now   recall      that the     forms
  $\tilde{\mathfrak{t}},\tilde{\mathfrak{h}}$ are generated    by the operators  $\tilde{W}, \tilde{W}_{\mathfrak{R}}$ respectively. Note that    in consequence of Lemmas \ref{L3.1}- \ref{L3.3}  these operators   are m-sectorial. Hence     by virtue  of    Theorem 2.9 \cite[p.326]{firstab_lit:kato1980}, we get  $T_{\tilde{\mathfrak{t}}}=\tilde{W},  T_{\tilde{\mathfrak{h}}}=\tilde{W}_{\mathfrak{R}} .$    Since we have proved that  $\mathfrak{Re}\,\tilde{\mathfrak{t}}=\tilde{\mathfrak{h}}$, then $T_{\mathfrak{Re}\,\tilde{\mathfrak{t}}}=\tilde{W}_{\mathfrak{R}}.$ Therefore   by definition   we have  that the operator $\tilde{W}_{\mathfrak{R}}$ is  the real part of the m-sectorial operator $\tilde{W},$   by symbol $\tilde{W}_{\mathfrak{R}}=Re\, \tilde{W}.$   Since we proved above  that $\tilde{W}_{\mathfrak{R}}$ has a compact resolvent, then
  using  Theorem 3.3 \cite[p.337]{firstab_lit:kato1980}  we conclude that the  operator $\tilde{W}$ has a compact resolvent. The proof corresponding   to the operator $\tilde{W}^{+}$ is absolutely  analogous.
\end{proof}

\begin{theorem}\label{T3.5}
The following two-sided estimate holds
\begin{equation}\label{3.9}
 \left\| S  \right\|^{-2}\lambda_{i}(R_{H})\leq \lambda_{i}\left(V\right)\leq  \left\|S^{- 1}      \right\|\, \lambda_{i}(R_{H}),\,i\in \mathbb{N} ,
\end{equation}
where  $H:=\tilde{W}_{\mathfrak{R}},$ $V:=\left(R_{\tilde{W}}\right)_{\mathfrak{R}},$ and $S$ is   a bounded   selfadjoint   operator   defined by  the  operator $W.$
 \end{theorem}
\begin{proof}
    It was shown in the proof of  Theorem \ref{T3.4} that $H=Re\, \tilde{W}.$   Hence  in consequence   of Lemma \ref{L3.1}, Lemma \ref{L3.2},   Theorem 3.2 \cite[p.337]{firstab_lit:kato1980}  there exist  the selfadjoint   operators  $B_{i}:=\{B_{i}\in\mathcal{B} (\mathfrak{H}),\,\|B_{i}\|\leq \tan \theta \},\,i=1,2 $
(where $\theta$ is the  semi-angle of the sector $\mathfrak{L}_{0}(\theta)\supset \Theta(\tilde{W})$) such  that
\begin{equation}\label{3.10}
\tilde{W}=H^{\frac{1}{2}}(I+i B_{1}) H^{\frac{1}{2}},\;\tilde{W}^{+}=H^{\frac{1}{2}}(I+i B_{2}) H^{\frac{1}{2}}.
\end{equation}
Since the set of  linear operators  generates  ring,   it follows that
\begin{equation*}
  Hf\! =\!\frac{1}{2}\left[H^{\frac{1}{2}}(I+i B_{1})  +H^{\frac{1}{2}}(I+i B_{2})\right]H^{\frac{1}{2}}\!  =
  $$
  $$
  =\! \frac{1}{2}\left\{H^{\frac{1}{2}}\left[(I+i B_{1})  + (I+i B_{2})\right]\right\}H^{\frac{1}{2}}\!=
$$
$$
= \! H f +
 \frac{i}{2} H^{\frac{1}{2}}\left(B_{1}+B_{2}\right)  H^{\frac{1}{2}}f  ,\;f\in \mathfrak{M}.
\end{equation*}
 Consequently
\begin{equation}\label{3.11}
H^{\frac{1}{2}}\left(B_{1}+B_{2}\right) H^{\frac{1}{2}}f=0  ,\;f\in \mathfrak{M}.
\end{equation}
Let us show that $B_{1}=-B_{2}.$
 In  accordance with  Lemma \ref{L3.3} the operator $H$ is m-accretive, hence we have
$
(H+\zeta)^{-1}\in  \mathcal{B }(\mathfrak{H}),\,\mathrm{Re}\,\zeta>0.
$
Using this fact, we get
\begin{equation}\label{3.12}
{\rm Re}\left([H+\zeta]^{-1}Hf,f\right)_{\mathfrak{H}}={\rm Re}\left([H+\zeta]^{-1}[H+\zeta]  f,f\right)_{\mathfrak{H}}-{\rm Re}\left(\zeta\,[H+\zeta]^{-1}   f,f\right)_{\mathfrak{H}}\geq
$$
$$
\geq \|f\|^{2}_{\mathfrak{H}}-|\zeta|\cdot\|(H+\zeta)^{-1}\|\cdot\|f\|^{2}_{\mathfrak{H}}=\|f\|^{2}_{\mathfrak{H}} \left(1-|\zeta|\cdot \|(H+\zeta)^{-1}\|\right),
$$
$$
\,\mathrm{Re}\,\zeta>0,\,f\in \mathrm{D}(H).
\end{equation}
Applying    inequality \eqref{3.6},   we obtain
$$
\| f\| _{\mathfrak{H}}\|(H+\zeta)^{-1}f\| _{\mathfrak{H}}\geq|(f,[H+\zeta]^{-1}f)|
 \geq ( {\rm Re} \zeta+C_{0} )\|(H+\zeta)^{-1}f\|^{2}_{\mathfrak{H}},\;f\in \mathfrak{H}.
$$
It implies that
$$
\|(H+\zeta)^{-1}\|\leq ({\rm Re}\zeta +C_{0}  )^{-1},\;{\rm Re}\zeta>0.
$$
  Combining    this  estimate  and  \eqref{3.12}, we have
$$
{\rm Re}\left([H+\zeta]^{-1}Hf,f\right)_{\!\mathfrak{H}}\geq \|f\|_{\mathfrak{H}}^{2}\left(1-\frac{|\zeta|}{{\rm Re}\zeta +C_{0}}\right),\,{\rm Re}\zeta>0,\,f\in \mathrm{D}(H).
$$
Applying    formula (3.45) \cite[p.282]{firstab_lit:kato1980} and  taking into account that $H^{\frac{1}{2}}$ is selfadjoint, we get
\begin{equation}\label{3.13}
 \left(H^{\frac{1}{2} }f,f\right)_{\!\mathfrak{H}}=\frac{1}{\pi}\int\limits_{0}^{\infty}\zeta^{-1/2}{\rm Re}\left([H+\zeta]^{-1}Hf,f\right)_{\!\mathfrak{H}}d\zeta\geq
$$
$$
 \geq\|f\|^{2}_{\mathfrak{H}}\cdot \frac{C_{0}}{\pi }\int\limits_{0}^{\infty}\frac{\zeta^{-1/2}}{\zeta+C_{0}}  d\zeta=
 \sqrt{C_{0}}  \|f\|^{2}_{\mathfrak{H}},\,f\in  \mathrm{D} (H) .
\end{equation}
Since in accordance with   Theorem 3.35 \cite[p.281]{firstab_lit:kato1980} the set      $\mathrm{D}(H)$ is  a core of the operator $H^{ \frac{1}{2}},$    then we can extend \eqref{3.13} to
  \begin{equation}\label{3.14}
 \left(H^{\frac{1}{2}}f,f\right)_{\!\!\mathfrak{H}}\geq \sqrt{C_{0}} \|f\|^{2}_{\mathfrak{H}},\;f\in \mathrm{D} (H^{\frac{1 }{2}}).
\end{equation}
Hence $\mathrm{N}(H^{\frac{1}{2} })=0.$ Combining this fact and  \eqref{3.11}, we obtain
\begin{equation}\label{3.15}
 \left(B_{1}+B_{2}\right)  H^{ \frac{1}{2}}f=0  ,\;f\in \mathfrak{M}.
\end{equation}
Let us show that the set $\mathfrak{M}$ is a core of the operator $H^{ \frac{1}{2}}.$ Note that due to Theorem 3.35 \cite[p.281]{firstab_lit:kato1980}
 the operator  $H^{ \frac{1}{2}}$
is selfadjoint and $\mathrm{D}(H)$ is a core of the  operator $H^{ \frac{1}{2}}.$ Hence   we have the representation
\begin{equation}\label{3.16}
 \|H^{ \frac{1}{2}}f \|^{2}_{\mathfrak{H}}=(Hf,f)_{ \mathfrak{H}},\,f\in \mathrm{D}(H).
\end{equation}
To achieve    our aim, it is
sufficient  to show  the following
\begin{equation}\label{3.17}
\forall\,f_{0}\in \mathrm{D}(H^{ \frac{1}{2}}),\,\exists\, \{f_{n}\}_{1}^{\infty}\subset \mathfrak{M}:\; f_{n}\stackrel{\mathfrak{H}}{\longrightarrow} f_{0},\,H^{ \frac{1}{2}}f_{n}\stackrel{\mathfrak{H}}{\longrightarrow} H^{ \frac{1}{2}}f_{0}.
 \end{equation}
Since  in accordance with the definition the set $\mathfrak{M}$ is a core of $H$, then we can extend   second relation \eqref{3.8} to
$
\sqrt{C_{0}} \|f\| _{\mathfrak{H}_{+}}\leq (Hf,f)_{\mathfrak{H}}\leq \sqrt{C_{4}} \|f\| _{\mathfrak{H}_{+}},\,f\in \mathrm{D}(H).
$
Applying    \eqref{3.16}, we can write
\begin{equation}\label{3.18}
\sqrt{C_{0}} \|f\| _{\mathfrak{H}_{+}}\leq \|H^{ \frac{1}{2}}f \| _{\mathfrak{H}}\leq \sqrt{C_{4}} \|f\| _{\mathfrak{H}_{+}},\,f\in \mathrm{D}(H).
 \end{equation}
Using    lower estimate \eqref{3.18} and the fact that  $\mathrm{D}(H)$    is a  core of    $H^{ \frac{1}{2}},$ it is not hard to prove that  $\mathrm{D}(H^{   \frac{1}{2} })\subset\mathfrak{H}_{+}.$
Taking into account this fact and using         upper estimate \eqref{3.18}, we obtain \eqref{3.17}. It implies that  $\mathfrak{M}$ is a core of $H^{ \frac{1}{2}}.$ Note that  in accordance with    Theorem 3.35   \cite[p.281]{firstab_lit:kato1980} the operator $ H ^{ \frac{1}{2}} $ is m-accretive. Hence  combining  Theorem 3.2 \cite[p.268]{firstab_lit:kato1980}        with \eqref{3.14},
  we obtain   $\mathrm{R}(H ^{ \frac{1}{2}})=\mathfrak{H}.$ Taking into account that      $\mathfrak{M}$ is a  core   of the operator $H^{ \frac{1}{2}},$  we   conclude that
$\mathrm{R}(\check{H} ^{\frac{1}{2} })$ is dense in $\mathfrak{H},$ where $\check{H} ^{ \frac{1}{2} }$ is the  restriction of  the operator $H ^{ \frac{1}{2}}$ to $\mathfrak{M}.$ Finally, by virtue of \eqref{3.15}, we have that the sum   $B_{1}+B_{2}$   equal to zero   on the dense subset  of $\mathfrak{H}.$ Since these operators are   defined on $\mathfrak{H}$ and  bounded,  then   $B_{1}=-B_{2}.$   Further,   we    use the denotation  $B_{1}:=B.$

Note that due to Lemma \ref{L3.2} there exist the  operators $R_{\tilde{W}},R_{\tilde{W}^{+}}.$      Using the  properties  of the operator $B,$    we get
$\|(I\pm iB)f\|_{\mathfrak{H} }\|f\|_{\mathfrak{H} }\geq\mathrm{Re }\left([I\pm iB]f,f\right)_{\mathfrak{H}} =\|f\|^{2}_{\mathfrak{H}},\,f\in \mathfrak{H}.$ Hence
\begin{equation*}
\|(I\pm iB)f\|_{\mathfrak{H} } \geq \|f\|_{\mathfrak{H}},\,f\in \mathfrak{H}.
\end{equation*}
 It implies that the operators  $I\pm iB$ are invertible.
Since it was proved above that  $ \mathrm{R} (H^{ \frac{1}{2}})=\mathfrak{H},\,\mathrm{N}(H^{\frac{1}{2}})=0$,    then  there exists an operator $H^{-\frac{1}{2}}$ defined on $\mathfrak{H}.$
  Using    representation   \eqref{3.10}  and  taking into account  the   reasonings given above,  we obtain
\begin{equation}\label{3.19}
R_{\tilde{W}}=H^{-\frac{1 }{2}}(I+iB )^{-1} H^{- \frac{1}{2}},\;R_{\tilde{W}^{+}}=H^{-\frac{1 }{2}}(I-iB )^{-1} H^{- \frac{1}{2}}.
\end{equation}
Note that the following equality can be proved easily  $R^{\ast}_{ \tilde{W}}=R^{\,}_{\tilde{W}^{+}}.$ Hence  we have
\begin{equation}\label{3.20}
V=\frac{1}{2}\left(R_{\tilde{W}}+R_{\tilde{W}^{+}}\right).
\end{equation}
Combining  \eqref{3.19},\eqref{3.20},  we get
\begin{equation}\label{3.21}
V=\frac{1}{2}\,H^{-\frac{1 }{2}}\left[(I+iB )^{-1}+(I-iB )^{-1} \right]H^{- \frac{1}{2}}.
\end{equation}
 Using the obvious identity
$
(I+B^{2})=(I+iB ) (I-iB )= (I-iB )(I+iB ),
$
 by  direct calculation we   get
\begin{equation}\label{3.22}
(I+iB )^{-1}+(I-iB )^{-1}=(I+B^{2})^{-1}.
\end{equation}
Combining \eqref{3.21},\eqref{3.22}, we obtain
\begin{equation}\label{3.23}
V=\frac{1}{2}\,H^{-\frac{1 }{2}}  (I+B^{2} )^{-1}  H^{- \frac{1}{2}}.
\end{equation}
Let us evaluate the form $\left(V  f,f\right)_{\mathfrak{H}}.$  Note  that  there exists the operator $R_{H}$ (see Lemma \ref{L3.3}). Since  $H$ is selfadjoint (see Lemma \ref{L3.3}), then  due to Theorem 3 \cite[p.136]{firstab_lit: Ahiezer1966}     $R_{H}$ is selfadjoint.  It is clear that $R_{H}$  is positive because  $H$ is positive.   Hence by virtue of  the  well-known theorem (see \cite[p.174]{firstab_lit:Krasnoselskii M.A.})   there exists a unique  square root of the operator $ R_{H} ,$ the   selfadjoint operator $ \hat{R} $ such that
  $\hat{R} \hat{R}  =R_{H}.$     Using the    decomposition $H=H^{\frac{1}{2}}H^{\frac{1}{2}},$   we get   $H^{-\frac{1}{2}}H^{-\frac{1}{2}}H=I.$ Hence
$R_{H}\subset H^{-\frac{1}{2}}H^{-\frac{1}{2}},$ but   $ \mathrm{D} (R_{H})=\mathfrak{H}.$ It implies that $R_{H}=H^{-\frac{1}{2}}H^{-\frac{1}{2}}.$  Using  the  uniqueness property  of    square root  we   obtain       $H^{-\frac{1}{2}}= \hat{R}.$    Let us use the shorthand notation  $S:=I+B^{2}.$  Note that due to the obvious inequality   $\left(\|Sf\|_{\mathfrak{H}}  \geq\|f\|_{\mathfrak{H}},\,f\in \mathfrak{H}\right)$ the  operator $S^{-1}$ is bounded on the set $\mathrm{R}(S).$ Taking into account the reasoning given above, we get
 $$
\left(V  f,f\right)_{\mathfrak{H}}=\left(H^{-\frac{1 }{2}} S^{-1}     H^{- \frac{1}{2}}   f,f\right)_{\mathfrak{H}}=
\left( S^{-1}     H^{- \frac{1}{2}}   f,H^{-\frac{1 }{2}}f\right)_{\mathfrak{H}}\leq
$$
$$
\leq \|S^{-1}     H^{- \frac{1}{2}}   f \| _{\mathfrak{H}}  \|H^{- \frac{1}{2}}   f \|_{\mathfrak{H}}\leq   \|S^{- 1}\|  \cdot  \|   H^{- \frac{1}{2}}   f \|^{2}_{\mathfrak{H}}
=
  \|S^{- 1}       \|\cdot\left(R_{H }  f,f\right)_{\mathfrak{H}},\;f\in \mathfrak{H}.
$$
On the other hand,  it is easy to see that  $(S^{-1}f,f)_{\mathfrak{H}}\geq \|S^{-1}f\|^{2}_{\mathfrak{H}},\,f\in \mathrm{R}(S).$ At the same time   it is obvious that   $S$ is bounded and we have   $\|S^{-1}f\|_{\mathfrak{H}}\geq \|S\|^{-1} \|f\|_{\mathfrak{H}},\,f\in \mathrm{R}(S).$ Using   these estimates, we have
$$
\left(V f,f\right)_{\mathfrak{H}}=\left( S^{-1}     H^{- \frac{1}{2}}   f,H^{-\frac{1 }{2}}f\right)_{\mathfrak{H}}\geq  \|S^{-1}     H^{- \frac{1}{2}}   f \|^{2}_{\mathfrak{H}}\geq
$$
$$
\geq  \| S   \|^{-2} \cdot  \|   H^{- \frac{1}{2}}   f \|^{2}_{\mathfrak{H}}= \| S   \|^{-2}  \cdot\left(R_{H }  f,f\right)_{\mathfrak{H}},\;f\in \mathfrak{H}.
$$
  Note that due to Theorem \ref{T3.4} the operator $R_{H}$ is compact. Combining \eqref{3.20} with Theorem \ref{T3.4}, we get that the operator $V$ is compact.  Taking into account these facts and   using   Lemma 1.1 \cite[p.45]{firstab_lit:1Gohberg1965}, we obtain
 \eqref{3.9}.
\end{proof}
 \begin{remark}
  Since  it was proved  above that  $R_{H}$  is selfadjoint and positive, then we have  $\lambda_{i}(R_{H})=s_{i} (R_{H}),\,i\in \mathbb{N}.$
 Note that in accordance with the facts established above  the operator $H:=\tilde{W}_{\mathfrak{R}}$ has a discrete spectrum and a compact resolvent.
   Due to   results represented   in  \cite{firstab_lit:Rosenblum}, \cite{firstab_lit:1Agranovich2013}, \cite{firstab_lit:fedosov1964}, we have an opportunity to obtain
   order of the operator $H$  in an  easy   way    in most particular  cases.
  \end{remark}
The following   theorem   is  formulated in   terms of      order   $\mu:=\mu(H)$    and  devoted to the Schatten-von Neumann classification of     the  operator     $R_{\tilde{W}}.$

\begin{theorem}\label{T3.7} We have the following classification

\begin{equation*}
R_{\tilde{W}}\in  \mathfrak{S}_{p},\,p= \left\{ \begin{aligned}
\!l,\,l>2/\mu,\,\mu\leq1,\\
   1,\,\mu>1    \\
\end{aligned}
 \right.\;.
\end{equation*}
Moreover   under  the  assumption   $ \lambda_{n}(R_{H})\geq  C \,n^{-\mu},\,n\in \mathbb{N},$  we have
$$
 R_{ \tilde{W}}\in\mathfrak{S}_{p}\;  \Rightarrow \;\mu p>1,\;1\leq p<\infty,
$$
where $\mu:=\mu(H).$
 \end{theorem}
\begin{proof}
  Consider the case $(\mu\leq1).$   Since we    already know  that   $R_{  \tilde{W }   }^{*}=R_{ \tilde{W}^{+} }^{\!},$ then it can easily be checked that
the operator $ R_{  \tilde{W }  }^{*}  R_{ \tilde{W }}^{\, } $   is a selfadjoint positive compact operator. Due to the well-known fact   \cite[p.174]{firstab_lit:Krasnoselskii M.A.} there exists the operator $|R_{\tilde{W }}  |.$
   By virtue  of Theorem 9.2 \cite[p.178]{firstab_lit:Krasnoselskii M.A.} the operator $|R_{\tilde{W }}  |$ is compact.
Since $\mathrm{N}(|R_{\tilde{W }}|^{2}) =0,$  it follows that $\mathrm{N}(|R_{\tilde{W }}|)=0.$ Hence      applying   Theorem \cite[p.189]{firstab_lit: Ahiezer1966}, we get that the operator
$|R_{\tilde{W }}|$ has an infinite  set of the eigenvalues.   Using  condition \eqref{2.3} (iii), we get
$$
{\rm Re}(R_{ \tilde{W}}f,f)_{\mathfrak{H}}\geq C_{0}\|R_{ \tilde{W}}f\|^{2}_{\mathfrak{H}},\;f\in \mathfrak{H}.
$$
Hence
$$
(|R_{\tilde{W }}|^{2} f,f)_{\mathfrak{H}}=\|R_{ \tilde{W}}f\|^{2}_{\mathfrak{H}}\leq C^{-1}_{0}{\rm Re}(R_{ \tilde{W}}f,f)_{\mathfrak{H}}= C^{-1}_{0}(V f,f)_{\mathfrak{H}},\,V:=  \left(R_{\tilde{W}}\right)_{\mathfrak{R}}.
$$
Since we already know that the operators $|R_{\tilde{W }}|^{2},V$ are compact, then  using  Lemma 1.1 \cite[p.45]{firstab_lit:1Gohberg1965}, Theorem \ref{T3.5},  we get
\begin{equation}\label{3.24}
 \lambda_{i} (|R_{\tilde{W }}|^{2} )\leq C^{-1}_{0} \,\lambda_{i}(V)\leq  C  i^{-\mu },\;i\in \mathbb{N}.
\end{equation}
   Recall  that by  definition  we have     $s_{i}(R_{ \tilde{W}})= \lambda_{ i }( |R_{\tilde{W }}|  ).$ Note that the operators $|R_{\tilde{W }}|,|R_{\tilde{W }}|^{2}$ have the  same eigenvectors.
This fact can be easily proved if we note   the obvious  relation
$
|R_{\tilde{W }}|^{2}f_{i}=|\lambda _{i} (|R_{\tilde{W }}|)|^{2} f_{i},\, i\in \mathbb{N}
$
and
   the  spectral representation for the  square root of a  selfadjoint positive compact operator
$$
|R_{\tilde{W }}|f=\sum\limits_{i=1}^{\infty}\sqrt{\lambda _{ i }(|R_{\tilde{W }}|^{2})}  \left(f,\varphi_{i}\right)  \varphi_{i}, \,f\in \mathfrak{H},
$$
where  $f_{i}\,, \varphi_{i}$ are  the  eigenvectors  of the  operators $|R_{\tilde{W }}|,|R_{\tilde{W }}|^{2}$ respectively  (see (10.25) \cite[p.201]{firstab_lit:Krasnoselskii M.A.}).  Hence
$ \lambda_{ i }( |R_{\tilde{W }}|  ) =
\sqrt{\lambda _{ i }( |R_{\tilde{W }}|^{2}  )} ,\,i\in \mathbb{N}.$
Combining this fact with  \eqref{3.24}, we get
$$
\sum\limits_{i=1}^{\infty}s^{p}_{i}(R_{ \tilde{W}})=
\sum\limits_{i=1}^{\infty}\lambda_{i}^{\frac{p}{2} }( |R_{\tilde{W }}|^{2}  )\leq C \sum\limits_{i=1}^{\infty} i^{-\frac{\mu p }{2}}.
$$
This completes the proof    for the case $(\mu\leq1).$

Consider the case $(\mu>1).$ It follows from
  \eqref{3.20} that  the  operator $V$ is positive     and   bounded. Hence   by virtue  of Lemma 8.1 \cite[p.126]{firstab_lit:1Gohberg1965}, we have that for any  orthonormal basis $\{\psi_{i}\}_{1}^{\infty}\subset \mathfrak{H}$ the following equalities  hold
\begin{equation}\label{3.25}
\sum\limits_{i=1}^{\infty}{\rm Re}(R_{\tilde{W}}\psi_{i},\psi_{i})_{\mathfrak{H}}=\sum\limits_{i=1}^{\infty}  (V \psi_{i},\psi_{i})_{\mathfrak{H}}=\sum\limits_{i=1}^{\infty}  (V\, \varphi_{i},\varphi_{i})_{\mathfrak{H}} ,
\end{equation}
where $\{\varphi_{i}\}_{1}^{\infty}$ is the orthonormal  basis of  the  eigenvectors of the operator $V.$
Due to  Theorem   \ref{T3.5}, we get
$$
  \sum\limits_{i=1}^{\infty}  (V \varphi_{i},\varphi_{i})_{\mathfrak{H}} =\sum\limits_{i=1}^{\infty}
 s_{i}(V) \leq C  \sum\limits_{i=1}^{\infty}  i^{-\mu }  .
$$
By virtue of  Lemma \ref{L3.1}, we get $ |{\rm Im} (R_{\tilde{W}}\psi_{i},\psi_{i})_{\mathfrak{H}}|\leq k^{-1} (\xi)\, {\rm Re}(R_{\tilde{W}}\psi_{i},\psi_{i})_{\mathfrak{H}}.$ Combining this fact with \eqref{3.25}, we get that  the    following series is convergent
$$
\sum\limits_{i=1}^{\infty} (R_{\tilde{W}}\psi_{i},\psi_{i})_{\mathfrak{H}}<\infty.
$$
Hence by definition \cite[p.125]{firstab_lit:1Gohberg1965} the operator $R_{\tilde{W}}$ has  a finite matrix trace.
 Using    Theorem 8.1 \cite[p.127]{firstab_lit:1Gohberg1965}, we get   $R_{ \tilde{W}}\in \mathfrak{S}_{1}.$ This completes the    proof for the case   $(\mu>1).$

Now, assume that  $ \lambda_{n}(R_{H})\geq  C \,n^{-\mu},\,n\in \mathbb{N},\,0\leq\mu<\infty.$     Let us show that the   operator $V$ has the complete orthonormal  system of the  eigenvectors.   Using   formula \eqref{3.23},   we get
$$
V^{\!^{-1}}=2H^ \frac{1 }{2}  (I+B^{2})      H^{ \frac{1}{2}},\; \mathrm{D} (V^{\!^{-1}})= \mathrm{R} (V).
$$
Let us prove that   $\mathrm{D}(V^{\!^{-1}})\subset \mathrm{D}(H).$    Note that   the set $\mathrm{D}(V^{^{\!-\!1}})$ consists  of the  elements $f+g,$ where $f\in \mathrm{D}(\tilde{W}),\,g\in \mathrm{D}(\tilde{W}^{+}).$    Using  representation \eqref{3.10}, it is easy to prove that
 $ \mathrm{D}(\tilde{W})\subset\mathrm{D}(H),\,\mathrm{D}(\tilde{W}^{+}) \subset\mathrm{D}(H) .$ This gives the desired result.
Taking into account the facts proven  above, we get
\begin{equation}\label{3.26}
 (V^{^{\!-\!1}}\!\!f,f)_{\mathfrak{H}}  =2(  S      H^{ \frac{1}{2}}f,H^ \frac{1 }{2}f)_{\mathfrak{H}}\geq 2\|H^ \frac{1 }{2} f\|^{2}_{\mathfrak{H}}=2(Hf,f)_{\mathfrak{H}}  ,\,f\in \mathrm{D}(V^{^{\!-\!1}}),
\end{equation}
where $S=I+B^{2}.$
 Since  $V$ is selfadjoint, then   due to  Theorem 3 \cite[p.136]{firstab_lit: Ahiezer1966} the operator    $V^{^{\!-\!1}}$ is selfadjoint.
  Combining  \eqref{3.26} with
 Lemma \ref{L3.3} we get that    $V^{^{\!-\!1}}$ is strictly accretive.
Using these facts   we can write
\begin{equation}\label{3.27}
\|f\| _{V^{^{ -\!1}}} \geq C\|f\|_{H  } ,\,f\in  \mathfrak{H}_{ V^{^{ -\!1}}} .
\end{equation}
Since  the operator $H$ has a discrete spectrum (see Theorem 5.3 \cite{firstab_lit:1kukushkin2018}), then any set  bounded   with respect to the      norm $\mathfrak{H}_{H}$ is a compact set   with respect to the norm    $\mathfrak{H}$ (see Theorem 4 \cite[p.220]{firstab_lit:mihlin1970}). Combining this fact with \eqref{3.27}, Theorem 3 \cite[p.216]{firstab_lit:mihlin1970}, we get
  that the operator $V^{^{\!-\!1}}$ has a discrete spectrum, i.e.   it has   the infinite set of  the eigenvalues
$\lambda_{1}\leq\lambda_{2}\leq...\leq\lambda_{i}\leq..., \, \lambda_{i} \rightarrow \infty  ,\,i\rightarrow \infty$ and the   complete orthonormal system of the  eigenvectors.
 Now  note that the operators $V,\,V^{^{\!-\!1}}$ have the same eigenvectors.  Therefore   the operator   $V$ has the  complete orthonormal  system of the  eigenvectors. Recall  that any  complete orthonormal system  is a  basis in  separable Hilbert space. Hence the complete orthonormal  system of the   eigenvectors of the operator $V$ is
  a basis in the space $\mathfrak{H}.$
 Let  $\{\varphi_{i}\}_{1}^{\infty}$ be the  complete orthonormal system  of the   eigenvectors of the  operator  $V$   and suppose    $
R_{ \tilde{W}}\in\mathfrak{S}_{p};$ then by virtue of   inequalities  (7.9) \cite[p.123]{firstab_lit:1Gohberg1965}, Theorem \ref{T3.5},  we get
$$
\sum\limits_{i=1}^{\infty}|s_{i} (R_{\tilde{W}} )|^{p}\geq\sum\limits_{i=1}^{\infty}| (R_{\tilde{W}}\varphi_{i},\varphi_{i})_{\mathfrak{H}}|^{p}\geq\sum\limits_{i=1}^{\infty}|{\rm Re}(R_{\tilde{W}}\varphi_{i},\varphi_{i})_{\mathfrak{H}}|^{p}=
$$
$$
=\sum\limits_{i=1}^{\infty}| (V \varphi_{i},\varphi_{i})_{\mathfrak{H}}|^{p}=\sum\limits_{i=1}^{\infty}
|\lambda_{i}(V)|^{p}\geq C   \sum\limits_{i=1}^{\infty}  i^{- \mu p }  .
$$
We claim   that $\mu p>1.$   Assuming the converse in the previous   inequality,  we come   to   contradiction with the condition  $R_{ \tilde{W}}\in\mathfrak{S}_{p}.$
This completes the proof.
\end{proof}

  The following theorem establishes  the  completeness property of  the  system  of   root vectors   of the operator $R_{\tilde{W}}.$

\begin{theorem}\label{T3.8}
Suppose  $\theta< \pi \mu/2;$   then the  system  of   root vectors of the operator $R_{ \tilde{W} }$ is complete, where $\theta$ is   the   semi-angle of the     sector $ \mathfrak{L}_{0}(\theta)\supset \Theta (\tilde{W}),\,\mu:=\mu(H).$
 \end{theorem}
\begin{proof}
Using    Lemma \ref{L3.1}, we have
\begin{equation}\label{3.28}
|{\rm Im} (R_{\tilde{W}}f,f)_{\mathfrak{H}}|\leq k^{-1}(\xi) \, {\rm Re}(R_{\tilde{W}}f,f)_{\mathfrak{H}},\,f\in \mathfrak{H}.
\end{equation}
Therefore   $ \overline{\Theta (R_{\tilde{W}})}\subset \mathfrak{L}_{0}(\theta).$  Note that the map
     $z: \mathbb{C}\rightarrow \mathbb{C},\;  z=1/\zeta$ takes each    eigenvalue  of the operator $R_{\tilde{W}}$  to the eigenvalue  of the operator $\tilde{W}.$ It is also clear    that   $z:\mathfrak{L}_{0}(\theta)\rightarrow \mathfrak{L}_{0}(\theta).$
  Using   the definition  \cite[p.302]{firstab_lit:1Gohberg1965} let us consider the following set
\begin{equation*}
   \mathfrak{P}:=\left\{z:\,z=t\, \xi,\,\xi\in  \overline{\Theta (R_{\tilde{W}})},\, 0\leq t<\infty\right\}.
\end{equation*}
It is easy to see  that $  \mathfrak{P}$ coincides with a closed sector of the complex plane  with the  vertex situated  at the point zero. Let us denote  by
 $\vartheta(R_{\tilde{W}})$ the  angle of this sector. It is obvious that $  \mathfrak{P}\subset \mathfrak{ L }_{0}(\theta).$ Therefore    $ 0 \leq\vartheta(R_{\tilde{W}}) \leq  2\theta.$ Let us prove that     $0 <\vartheta(R_{\tilde{W}}),$ i.e.    the strict inequality holds.    If we assume  that $\vartheta(R_{\tilde{W}})=0,$ then we get  $ e^{-i \mathrm{arg} z} =\varsigma,\,\forall z\in  \mathfrak{P}\setminus 0, $
 where $\varsigma$  is a constant independent on $z.$ In consequence of this fact   we have $\mathrm{Im}\, \Theta (\varsigma R_{\tilde{W}})=0.$ Hence the operator $\varsigma R_{\tilde{W}}$ is symmetric (see Problem  3.9 \cite[p.269]{firstab_lit:kato1980}) and by virtue of the fact   $ \mathrm{D} (\varsigma R_{\tilde{W}})=\mathfrak{H}$ one is selfadjoint. On the other hand, taking into account   the      equality  $R^{\ast}_{ \tilde{W}}=R^{\,}_{ \tilde{W}^{+}}$   (see the proof of   Theorem \ref{T3.5}), we have $(\varsigma R_{\tilde{W}}f,g)_{\mathfrak{H}}=( f,\bar{\varsigma} R_{\tilde{W}^{+}}g)_{\mathfrak{H}},\,  f,g\in \mathfrak{H}.$ Hence
 $\varsigma R_{\tilde{W}}=\bar{\varsigma} R_{\tilde{W}^{+}}.$ In the particular case we have $\forall f\in \mathfrak{H},\,\mathrm{Im}f=0:\, \mathrm{Re}\,\varsigma\,R_{\tilde{W}}f= \mathrm{Re}\,\varsigma\,R_{\tilde{W}^{+}}f,\,\mathrm{Im}\,\varsigma\,R_{\tilde{W}}f= -\mathrm{Im}\,\varsigma\,R_{\tilde{W}^{+}}f  .  $ It implies that  $\mathrm{N}(R_{\tilde{W}})\neq 0.$  This contradiction   concludes  the proof of the fact $ \vartheta(R_{\tilde{W}})>0.$
    Let us use   Theorem 6.2 \cite[p.305]{firstab_lit:1Gohberg1965}    according to  which   we have the following. If the following two conditions  (a) and (b) are fulfilled, then   the  system  of   root vectors  of the operator $R_{\tilde{W}}$ is complete.  \\

\noindent  a) $\vartheta(R_{\tilde{W}})=  \pi/ d ,$ where $d>1,$\\

\noindent b) for some $\beta,$ the  operator $B:=\left(e^{ i \beta}R_{\tilde{W}} \right)_{\mathfrak{I}}:\;s_{i}(B)=o(i^{-1/d }),\,i\rightarrow \infty .$\\

\noindent Let us show that conditions (a) and (b) are fulfilled. Note that due to Lemma \ref{L3.1} we have  $0\leq\theta<\pi/2.$
Hence $ 0<\vartheta(R_{\tilde{W}})<\pi.$ It implies that  there exists  $1<d<\infty$ such that
$\vartheta(R_{\tilde{W}})=  \pi/ d.$  Thus   condition  (a) is fulfilled.
  Let us choose the  certain value  $\beta=\pi/2$ in   condition (b) and notice  that $\left(e^{ i \pi/2}R_{\tilde{W}} \right)_{\mathfrak{I}}=\left( R_{\tilde{W}} \right)_{\mathfrak{R}}. $ Since the operator $V:=\left( R_{\tilde{W}} \right)_{\mathfrak{R}}$ is  selfadjoint, then we have   $s_{i}(V)=\lambda_{i}(V),\,i\in \mathbb{N}.$
  In consequence of Theorem \ref{T3.5}, we obtain
\begin{equation*}
 s_{i}(V)\, i^{1/d} = s_{i}(V)\, i^{ \mu} \cdot i^{1/d-\mu}    \leq C\cdot i^{1/d-\mu} ,\,i\in \mathbb{N}.
\end{equation*}
 Hence to achieve  condition  (b),      it is sufficient  to show that  $d>\mu^{-1}.$
By virtue of the conditions $ \vartheta(R_{\tilde{W}}) \leq 2 \theta,\,\theta< \pi \mu/2,$ we have $d=\pi/\vartheta(R_{\tilde{W}})\geq \pi/2\theta>\mu^{-1}.$ Hence we obtain  $s_{i}(V)=o(i^{-1/d }).$
 Since   both     conditions (a),(b) are fulfilled, then using   Theorem 6.2 \cite[p.305]{firstab_lit:1Gohberg1965} we complete the proof.
\end{proof}

  Proven Theorem \ref{T3.7} is  devoted to    the  description  of $s$-numbers    behavior  but     questions related with   asymptotic of the eigenvalues $\lambda_{i}(R_{\tilde{W}}),\,i\in \mathbb{N}$ are still relevant in our work.  It is a well-known fact  that for any  bounded  operator with the compact  imaginary component    there is a relationship between $s$-numbers of the  imaginary component  and the  eigenvalues    (see \cite {firstab_lit:1Gohberg1965}).  Similarly using  the information on    $s$-numbers of the real component,    we  can obtain  an asymptotic formula for the eigenvalues $\lambda_{i}(R_{\tilde{W}}),\,i\in \mathbb{N}.$      This idea is realized in the following theorem.
\begin{theorem}\label{T3.9}
 The following inequality   holds
 \begin{equation}\label{3.29}
\sum\limits_{i=1}^{n}|\lambda_{i}(R_{\tilde{W}})|^{p}\leq
  \sec ^{p} \theta \,\left\|S^{- 1}      \right\| \sum\limits_{i=1}^{n } \,\lambda^{ p}_{i}(R_{H}),
\end{equation}
$$
(n=1,2,...,\, \nu(R_{\tilde{W}})),\,1\leq p<\infty.
$$
Moreover   if  $\nu(R_{\tilde{W}})=\infty$ and the  order $\mu(H)\neq0,$  then  the following asymptotic formula  holds
\begin{equation}\label{3.30}
|\lambda_{i}(R_{\tilde{W}})|=  o\left(i^{-\mu+\varepsilon}    \right)\!,\,i\rightarrow \infty,\;\forall \varepsilon>0.
\end{equation}
\end{theorem}

\begin{proof}
Let $L$ be a bounded    operator with  a compact imaginary component.  Note that  according to Theorem 6.1 \cite[p.81]{firstab_lit:1Gohberg1965},   we have
 \begin{equation}\label{3.31}
\sum\limits_{m=1}^{k}|{\rm Im}\,\lambda_{m}(L)|^{p}\leq \sum\limits_{m=1}^{k}
|s_{m} (L_{\mathfrak{I} } )|^{p},\;\left(k=1,\,2,...,\,\nu_{\,\mathfrak{I}}(L)\right),\,1\leq p<\infty,
\end{equation}
where    $ \nu_{\,\mathfrak{I}}(L) \leq \infty $ is   the sum of   all  algebraic multiplicities corresponding to   the not real  eigenvalues   of the operator $L$ (see  \cite[p.79]{firstab_lit:1Gohberg1965}).    It can easily be checked that
 \begin{equation}\label{3.32}
(iL)_{\mathfrak{I}}=L_{\mathfrak{R}},\; \mathrm{Im }\, \lambda_{m} (i\,L)= \mathrm{Re}\,  \lambda_{m}(L),\;m\in \mathbb{N}.
\end{equation}
By virtue of  \eqref{3.28}, we have ${\rm Re}\,\lambda_{m}(R_{\tilde{W}})>0,\,m=1,2,...,\nu\left(  R_{\tilde{W}} \right) .$ Combining this fact with     \eqref{3.32}, we get
 $\nu_{\mathfrak{I}}(i R_{\tilde{W}} )=\nu\left(R_{\tilde{W}}\right).$ Taking  into account the previous  equality  and    combining \eqref{3.31},\eqref{3.32},       we obtain
 \begin{equation}\label{3.33}
\sum\limits_{m=1}^{k}|{\rm Re}\,\lambda_{m}(R_{\tilde{W}})|^{p}\leq \sum\limits_{m=1}^{k}
|s_{m} (V )|^{p} ,\; \left(k=1,\,2,...\,,\nu  (R_{\tilde{W}})\right),\,V:= \left(R_{\tilde{W}}\right)_{\mathfrak{R}}.
\end{equation}
 Note that by virtue of \eqref{3.28}, we have
$$
|{\rm Im}\,\lambda_{m}(R_{\tilde{W}})|\leq  \tan \theta \,{\rm Re}\,\lambda_{m}\,(R_{\tilde{W}}),\,m\in \mathbb{N}.
$$
Hence
 \begin{equation}\label{3.34}
| \lambda_{m}(R_{\tilde{W}})|= \sqrt{|{\rm Im}\,\lambda_{m}(R_{\tilde{W}})|^{2}+|{\rm Re}\,\lambda_{m}(R_{\tilde{W}})|^{2}}\leq
$$
$$
  \leq \sqrt{\tan^{2}\theta+1} \;|{\rm Re}\,\lambda_{m}(R_{\tilde{W}})|
=\sec \theta\, |{\rm Re}\,\lambda_{m}(R_{\tilde{W}})|,\,m\in \mathbb{N}.
\end{equation}
Combining \eqref{3.33},\eqref{3.34}, we get
 \begin{equation*}
\sum\limits_{m=1}^{k}| \lambda_{m}(R_{\tilde{W}})|^{p}\leq \sec^{p}\!\theta\,\sum\limits_{m=1}^{k}
|s_{m} (V )|^{p} ,\; \left(k=1,\,2,...\,,\nu  (R_{\tilde{W}})\right).
\end{equation*}
Using   \eqref{3.9}, we complete     the proof of   inequality \eqref{3.29}.

Suppose $  \nu(R_{\tilde{W}})=\infty,\, \mu(H)\neq0 $ and let us prove \eqref{3.30}.   Note   that for    $\mu>0$ and for any $ \varepsilon>0,$ we can choose  $p$  so that
 $\mu p>1,\, \mu-\varepsilon<1/p .$
Using the condition $\mu p>1,$  we obtain  convergence of the series on  the left side of    \eqref{3.29}. It implies that
 \begin{equation}\label{3.35}
|\lambda_{i}(R_{\tilde{W}})| i^{1/p}\rightarrow 0,\;i\rightarrow\infty.
\end{equation}
It is obvious that
$$
|\lambda_{i}(R_{\tilde{W}})| i^{ \mu-\varepsilon}<|\lambda_{i}(R_{\tilde{W}})| i^{ 1/p} ,\,i\in \mathbb{N}.
$$
Taking into account \eqref{3.35}, we obtain  \eqref{3.30}.
\end{proof}

\section{Applications}

\noindent {\bf 1.} We begin with definitions.
Suppose  $\Omega$
is a convex domain of the n-dimensional Euclidian space with the  sufficient smooth
boundary,
$L_{2}(\Omega)$ is  a  complex Lebesgue space  of summable with square functions,   $H^{2}(\Omega),$ $H^{1}(\Omega)$
 are  complex Sobolev spaces, $D_{i}f:= \partial f/\partial x_{i} ,\,1 \leq i \leq n$ is the    weak   partial derivatives of  of the function $f.$   Consider a sum of a  uniformly elliptic operator
and the extension of the Kipriyanov fractional differential operator of   order $0 <\alpha < 1$ (see Lemma
2.5 \cite{firstab_lit:1kukushkin2018})
\begin{equation*}
 Lu:=-  D_{j} ( a^{ij} D_{i}f) \, +   \mathfrak{D}  ^{ \alpha }_{0+}f ,
 $$
 $$
   \mathrm{D}(L)=H^{2}(\Omega)\cap H^{1}_{0}(\Omega),
  \end{equation*}
 with the following  assumptions relative to the  real-valued   coefficients
\begin{equation*}
 a^{ij}(Q) \in C^{1}(\bar{\Omega}),\,   a^{ij}\xi _{i}  \xi _{j}  \geq   a  |\xi|^{2} ,\,  a  >0 .
 \end{equation*}
 It was proved in the  paper  \cite{firstab_lit:1kukushkin2018} that the operator
 $   L^{+}f:=-D_{i} ( a^{ij} D_{j}f)+ \mathfrak{D}   ^{ \alpha }_{d-}f,$
 $
 \mathrm{D}   (L^{+})=\mathrm{D}(L)
 $  is formal adjoint with respect  to $L.$
Note that  in accordance with Theorem 2    \cite{firstab_lit(JFCA)kukushkin2018} we have $\mathrm{R}(L) = \mathrm{R}(L^{+}) = L_{2}(\Omega),$ due to     Theorem 4.2 \cite{firstab_lit:1kukushkin2018} the operators
 $L,L^{+}$ are strictly accretive. Taking into account  these facts we can conclude that the operators $L,L^{+}$ are closed
(see problem 5.15 \cite[p.165]{firstab_lit:kato1980}). Consider the operator $L_{\mathfrak{R}}.$  Having made the   absolutely analogous
reasonings   as in the previous case,  we  conclude that the operator $L_{\mathfrak{R}}$ is closed. Applying
the   reasonings of  Theorem 4.3 \cite{firstab_lit:1kukushkin2018}, we obtain that the operator  $L_{\mathfrak{R}}$ is   selfadjoint and strictly accretive.
Recall  that to apply    the methods    described  in the paper  \cite{firstab_lit:Shkalikov A.} we must have some decomposition of the initial
operator $L$ on a sum of the main part   and    the  operator-perturbation, where
the main part must be an  operator of a special type   either a  selfadjoint or  a normal operator.
Note that a uniformly elliptic operator of  second order  is
neither selfadjoint  no normal  in general case.   To demonstrate  the   significance of the method obtained in this paper,   we would like to note  that    a search  for  a convenient decomposition of $L$ on a sum
of a selfadjoint operator and the  operator-perturbation does not seem to be a reasonable   way. Now  to justify this  claim  we consider one of  possible decompositions of $L$ on  a sum.   Consider  a selfadjoint strictly accretive operator
$  \mathcal{T}  :\mathfrak{H}\rightarrow \mathfrak{H}. $

\begin{definition}
In accordance with the   definition of the paper  \cite{firstab_lit:Shkalikov A.},  a quadratic form $\mathfrak{a}: = \mathfrak{a}[f]$ is called a $\mathcal{T}$ - subordinated form if the following condition holds
\begin{equation}\label{4.1}
\left|\mathfrak{a}[f]\right| \leq b\, \mathfrak{t}[f] +M\|f\|^{2}_{\mathfrak{H}},\;
  \mathrm{D}(\mathfrak{a}) \supset \mathrm{D}(\mathfrak{t}),\; b < 1,\; M > 0,
\end{equation}
where $\mathfrak{t}[f]=\|\mathcal{T}^{\frac{1}{2}}\|^{2}_{\mathfrak{H}},\,f\in \mathrm{D}(\mathcal{T}^{\frac{1}{2}}).$
 The form $\mathfrak{a}$  is called  a completely $\mathcal{T}$ - subordinated form    if besides   of \eqref{4.1}  we have the following additional condition  $\forall \varepsilon>0 \,\exists b,M>0:\, b <\varepsilon.$
\end{definition}
Let us consider the trivial  decomposition  of the operator $L$ on the sum $L=2 L_{\mathfrak{R}}-L^{+}$ and let us use the notation $\mathcal{T}:=2L_{\mathfrak{R}},\, \mathcal{A}:=-L^{+}.$ Then we have $L=\mathcal{T}+\mathcal{A}.$
Due to the sectorial property
proven in   Theorem 4.2 \cite{firstab_lit:1kukushkin2018}  we have
\begin{equation}\label{4.2}
 |(\mathcal{A}f,f)_{L_{2} }| \!=\!\sec \theta_{f}\, |\mathrm{Re}(\mathcal{A}f,f)_{L_{2} }|\! = \!\sec \theta_{f}\, \frac{1}{2}(\mathcal{T}f,f)_{L_{2} },\,f\in \mathrm{D}(\mathcal{T}),
\end{equation}
 where $ 0\leq\theta_{f}\leq \theta,\, \theta_{f}:=\left|\mathrm{arg}  (L^{+}f,f)_{L_{2} }\right|,\, L_{2}:=L_{2}(\Omega)  $   and $\theta$ is the semi-angle   corresponding to the sector $\mathfrak{L}_{0}(\theta).$
  Due to Theorem 4.3 \cite{firstab_lit:1kukushkin2018}  the operator $\mathcal{T}$ is m-accretive. Hence in consequence of Theorem
3.35  \cite[p.281]{firstab_lit:kato1980} we have that  $\mathrm{D}(\mathcal{T} )$   is a core of the operator $\mathcal{T}^{\frac{1}{2}}.$ It implies that we can extend   relation \eqref{4.2} to
\begin{equation}\label{4.3}
 \frac{1}{2} \, \mathbf{\mathfrak{t}}[f]\leq|\mathbf{\mathfrak{a }}[f]|\leq \sec \theta \frac{1}{2} \, \mathbf{\mathfrak{t}}[f],\, f\in \mathrm{D} (\mathbf{\mathfrak{t}}),
\end{equation}
where $\mathfrak{a}$ is a quadratic form generated by $\mathcal{A}$ and $\mathfrak{t}[f] = \|\mathcal{T}^{\frac{1}{2}}f\|^{2}_{\mathfrak{H}}.$
 If we consider
the case $0<\theta<\pi/3,$ then it is obvious  that there exist  constants  $b < 1$ and $M > 0$ such that
the following inequality holds
$$
|\mathfrak{a}[f]| \leq b\, \mathfrak{t}[f] +M\|f\|^{2}_{L_{2}},\;
 f\in \mathrm{D}(\mathfrak{t}).
$$
Hence the form $\mathfrak{a}$ is a $\mathcal{T}$ - subordinated form.
In accordance with  the  definition given in the paper \cite{firstab_lit:Shkalikov A.} it means $\mathcal{T}$ - subordination of the operator $\mathcal{A}$  in the sense of   form. Assume
that $\forall \varepsilon>0\,\exists b,M>0:\,  b <\varepsilon.$  Using  inequality \eqref{4.3},  we get
$$
   \frac{1}{2}\mathfrak{t}[f]   \leq \varepsilon \, \mathfrak{t}[f] + M(\varepsilon)\|f\|^{2}_{L_{2} };\; \mathfrak{t}[f]\leq \frac{2 M(\varepsilon)}{(1-2\varepsilon)}\|f\|^{2}_{L_{2} },  \,f\in  \mathrm{D}(\mathfrak{t}),\, \varepsilon < 1/2.
$$
Using the strictly accretive property of the operator $L$ (see inequality  (4.9) \cite{firstab_lit:1kukushkin2018}), we   obtain
$$
\|f\|^{2}_{H^{1}_{0} } C \leq  \mathfrak{t}[f],  \,f\in  \mathrm{D}(\mathfrak{t}).
$$
On the other hand, using the results of the paper  \cite{firstab_lit:1kukushkin2018},  it is easy to prove that $H^{1}_{0}(\Omega)\subset\mathrm{D}(\mathfrak{t}).$
Taking into account the facts considered above,   we get
$$
\|f\|_{H^{1}_{0} }   \leq  C \|f\|_{L_{2}},  \,f\in H^{1}_{0}(\Omega)  .
$$
    It cannot be! It is a well-known fact. This contradiction shows us that the form $\mathfrak{a}$ is not    a  completely $\mathcal{T}$ - subordinated form. It implies that we cannot  use     Theorem 8.4 \cite{firstab_lit:Shkalikov A.}   which could give us an opportunity to describe the spectral properties of the operator  $L.$
Note that the reasonings corresponding to  another trivial decomposition of $L$ on a sum  is analogous.

     This rather particular example does not aim to show the inability of using   remarkable
methods considered in the paper  \cite{firstab_lit:Shkalikov A.}  but only creates   prerequisite for    some   value  of    another
method  based on  using spectral properties of the real component of the  initial operator $L.$ Now we would like to demonstrate   the    effectiveness    of   this method. Suppose   $\mathfrak{H} := L_{2}(\Omega
),\, \mathfrak{H}^{+} := H^{1}_{0}(\Omega),\, Tf :=  -D_{j}(a^{ij}D_{i}f),\, Af :=
   \mathfrak{D}^{ \alpha }_{0+}f,\; \mathrm{D}(T),\mathrm{D}(A) =H^{2}(\Omega)\cap H_{0}^{1}(\Omega);$ then    due to the Rellich-Kondrachov  theorem we have that  condition
\eqref{2.1} is fulfilled. Due to  the  results obtained in the paper  \cite{firstab_lit:1kukushkin2018} we have that   condition  \eqref{2.3} is  fulfilled. Applying the results obtained in the paper  \cite{firstab_lit:1kukushkin2018}
we  conclude that the operator $L_{\mathfrak{R}}$ has   non-zero order. Hence we can apply the abstract results of
this paper to the   operator $L.$ In fact, Theorems \ref{T3.7}-\ref{T3.9} describe the  spectral properties of the operator $L.$

\noindent{\bf 2.}   We deal with the differential operator acting in the complex Sobolev  space and defined by the following
expression
$$
\mathcal{L}f := (c_{k}f^{(k)})^{(k)} + (c_{k-1}f^{(k-1)})^{(k-1)}+...+  c_{0}f,
$$
$$
\mathrm{D}(\mathcal{L}) = H^{2k}(I)\cap H_{0}^{k}(I),\,k\in \mathbb{N},
$$
where    $I: = (a, b) \subset \mathbb{R},$ the   complex-valued coefficients
$c_{j}(x)\in C^{(j)}(\bar{I})$ satisfy the condition $  {\rm sign} (\mathrm{Re} c_{j}) = (-1)^{j} ,\, j = 1, 2, ..., k.$
It is easy
to see that
$$
\mathrm{Re}(\mathcal{L}f,f)_{L_{2}(I)}\geq\sum\limits_{j=0}^{k}|\mathrm{Re} c_{j}|\,\|f^{(j)}\|^{2}_{L_{2}(I)}\geq C \|f^{(j)}\|^{2}_{H_{0}^{k}(I)},\;f\in \mathrm{D}(\mathcal{L}).
$$
On the other hand
$$
|(\mathcal{L}f,f)_{L_{2}(I)}|=\left|\sum\limits_{j=0}^{k}(-1)^{j}(c_{j}f^{(j)},g^{(j)} )_{L_{2}(I)}\right|\leq
\sum\limits_{j=0}^{k}\left|(c_{j}f^{(j)},g^{(j)} )_{L_{2}(I)}\right|\leq
$$
$$
\leq C \sum\limits_{j=0}^{k} \|f^{(j)}\| _{L_{2}(I)}\|g^{(j)}\| _{L_{2}(I)}\leq
\|f \| _{H^{k}_{0}(I)}\|g \| _{H^{k}_{0}(I)},\;f\in \mathrm{D}(\mathcal{L}).
$$
Consider the  Riemann-Liouville   operators of fractional differentiation of   arbitrary non-negative
order $\alpha$ (see \cite[p.44]{firstab_lit:samko1987})  defined by the expressions
\begin{equation*}
 D_{a+}^{\alpha}f=\left(\frac{d}{dx}\right)^{[\alpha]+1}\!\!\!\!I_{a+}^{1-\{\alpha\}}f;\;
 D_{b-}^{\alpha}f=\left(-\frac{d}{dx}\right)^{[\alpha]+1}\!\!\!\!I_{b-}^{1-\{\alpha\}}f,
\end{equation*}
where the fractional integrals of      arbitrary positive order  $\alpha$ defined by
$$
\left(I_{a+}^{\alpha}f\right)\!(x)=\frac{1}{\Gamma(\alpha)}\int\limits_{a}^{x}\frac{f(t)}{(x-t)^{1-\alpha}}dt,
 \left(I_{b-}^{\alpha}f\right)\!(x)=\frac{1}{\Gamma(\alpha)}\int\limits_{x}^{b}\frac{f(t)}{(t-x)^{1-\alpha}}dt
, f\in L_{1}(I).
$$
Suppose  $0<\alpha<1,\, f\in AC^{l+1}(\bar{I}),\,f^{(j)}(a)=f^{(j)}(b)=0,\,j=0,1,...,l;$ then the next formulas follows
from   Theorem 2.2 \cite[p.46]{firstab_lit:samko1987}
\begin{equation}\label{4.4}
 D_{a+}^{\alpha+l}f= I_{a+}^{1- \alpha }f^{(l+1)},\;
 D_{b-}^{\alpha+l}f= (-1)^{l+1}I_{b-}^{1- \alpha }f^{(l+1)}.
\end{equation}
  Further, we need  the following inequalities    (see  \cite{firstab_lit:(vlad)kukushkin2016})
\begin{equation}\label{4.5}
\mathrm{Re} (D_{a+}^{\alpha}f,f)_{L_{2}(I)}\geq C\|f\|^{2}_{L_{2}(I)},\,f\in I_{a+}^{\alpha}(L_{2}),\;
$$
$$
\mathrm{Re} (D_{b-}^{\alpha}f,f)_{L_{2}(I)}\geq C\|f\|^{2}_{L_{2}(I)},\,f\in I_{b-}^{\alpha}(L_{2}),
\end{equation}
where $I_{a+}^{\alpha}(L_{2}),I_{b-}^{\alpha}(L_{2})$ are the  classes of  the  functions representable by the fractional integrals (see\cite{firstab_lit:samko1987}).
  Consider the following operator  with the  constant  real-valued  coefficients
$$
\mathcal{D}f:=p_{n}D_{a+}^{\alpha_{n}}+q_{n}D_{b-}^{\beta_{n}}+p_{n-1}D_{a+}^{\alpha_{n-1}}+q_{n-1}D_{b-}^{\beta_{n-1}}+...+
p_{0}D_{a+}^{\alpha_{0}}+q_{0}D_{b-}^{\beta_{0}},
$$
$$
\mathrm{D}(\mathcal{D}) = H^{2k}(I)\cap H_{0}^{k}(I),\,n\in \mathbb{N},
$$
where $\alpha_{j},\beta_{j}\geq 0,\,0 \leq [\alpha_{j}],[\beta_{j}] < k,\, j = 0, 1, ..., n.,\;$
\begin{equation*}
 q_{j}\geq0,\;{\rm sign}\,p_{j}= \left\{ \begin{aligned}
  (-1)^{\frac{[\alpha_{j}]+1}{2}},\,[\alpha_{j}]=2m-1,\,m\in \mathbb{N},\\
\!\!\!\!\!\!\! \!\!\!\!(-1)^{\frac{[\alpha_{j}]}{2}},\;\,[\alpha_{j}]=2m,\,\,m\in \mathbb{N}_{0}   .\\
\end{aligned}
\right.
\end{equation*}
Using \eqref{4.4},\eqref{4.5},  we get
$$
(p_{j}D_{a+}^{\alpha_{j}}f,\!f)_{L_{2}(I)}\!=
\!p_{j}\!\left(\!\!\left(\frac{d}{dx}\right)^{\!\!\!m}\!\!D_{a+}^{m-1+\{\alpha_{j}\}}\!\!f,\!f   \!\right)_{\!\!L_{2}(I)}\!\!\!\!\!\! =
(\!-1)^{m}p_{j}\!\left(\! I_{a+}^{1-\{\alpha_{j}\}}\!\!f^{(m)}\!\!,\!f^{(m)}   \!\right)_{\!\!L_{2}(I)}\!\! \geq
$$
$$
\geq C\left\|I_{a+}^{1-\{\alpha_{j}\}}f^{(m)}\right\|^{2}_{L_{2}(I)}=
C\left\|D_{a+}^{\{\alpha_{j}\}}f^{(m-1)}\right\|^{2}_{L_{2}(I)}\geq C \left\| f^{(m-1)}\right\|^{2}_{L_{2}(I)},
$$
 where  $f \in \mathrm{D}(\mathcal{D})$ is    a real-valued function and   $ [\alpha_{j}]=2m-1,\,m\in \mathbb{N}.$
  Similarly,  we obtain for orders  $ [\alpha_{j}]=2m,\,m\in \mathbb{N}_{0}$
$$
(p_{j}D_{a+}^{\alpha_{j}}f,f)_{L_{2}(I)}=p_{j}\left( D_{a+}^{2m +\{\alpha_{j}\}}f,f   \right)_{L_{2}(I)}=(-1)^{m}p_{j}\left( D_{a+}^{m+\{\alpha_{j}\}}f ,f^{(m)}   \right)_{L_{2}(I)}=
$$
$$
=(-1)^{m}p_{j}\left( D_{a+}^{ \{\alpha_{j}\}}f^{(m)} ,f^{(m)}   \right)_{\!L_{2}(I)}\geq C \left\| f^{(m)}\right\|^{2}_{L_{2}(I)}.
$$
Thus in both cases  we have
$$
(p_{j}D_{a+}^{\alpha_{j}}f,f)_{L_{2}(I)}\geq C \left\| f^{(s)}\right\|^{2}_{L_{2}(I)},\;s= \big[[\alpha_{j}]/2\big] .
$$
 In the same way, we obtain the inequality
$$
(q_{j}D_{b-}^{\alpha_{j}}f,f)_{L_{2}(I)}\geq C \left\| f^{(s)}\right\|^{2}_{L_{2}(I)},\;s= \big[[\alpha_{j}]/2\big] .
$$
  Hence in the
complex case we have
$$
\mathrm{Re}(\mathcal{D}f,f)_{L_{2}(I)}\geq C \left\| f \right\|^{2}_{L_{2}(I)},\;f\in \mathrm{D}(\mathcal{D}).
$$
Combining   Theorem 2.6 \cite[p.53]{firstab_lit:samko1987}  with  \eqref{4.4}, we get
$$
\left\| p_{j}D_{a+}^{\alpha_{j}}f \right\| _{L_{2}(I)}=  \left\|  I_{a+}^{1-\{\alpha_{j}\}}f^{([\alpha_{j}]+1)} \right\| _{L_{2}(I)}
 \leq C   \left\|   f^{([\alpha_{j}]+1)} \right\|_{L_{2}(I)}\leq C   \left\|   f  \right\|_{H^{k}_{0}(I)};
 $$
 $$
 \;\left\|q_{j}D_{b-}^{\alpha_{j}}f \right\| _{L_{2}(I)}
\leq  C   \left\|   f  \right\|_{H^{k}_{0}(I)},\;f\in \mathrm{D}(\mathcal{D}).
$$
 Hence, we obtain
$$
\left\| \mathcal{D}f \right\| _{L_{2}(I)}\leq C \left\|  f \right\|_{H^{k}_{0}(I)},\;f\in \mathrm{D}(\mathcal{D}).
$$
Now we can formulate the main result. Consider the operator
$$
G=\mathcal{L}+\mathcal{D},
$$
$$
\mathrm{D}(G)=H^{2k}(I)\cap H_{0}^{k}(I).
$$
Suppose  $\mathfrak{H} := L_{2}(I),\, \mathfrak{H}^{+} := H_{0}^{k}(I),\, T := \mathcal{L},\, A := \mathcal{D};$ then due to the well-known fact of the  Sobolev spaces theory
      condition \eqref{2.1} is fulfilled, due to the reasonings given above condition  \eqref{2.3}
is fulfilled. Taking into account the equality
$$
\mathcal{L}_{\mathfrak{R}}f = ( \mathrm{Re}c_{k} f^{(k)})^{(k)} + ( \mathrm{Re}c_{k-1} f^{(k-1)})^{(k-1)} + ... +  \mathrm{Re}c_{0} f,\;f\in \mathrm{D}(\mathcal{D})
$$
and using  the method described
  in the paper  \cite{firstab_lit:2kukushkin2017}, we   can prove   that the operator $\tilde{G}_{\mathfrak{R}}$  has   non-zero order.
Hence we can successfully  apply  the abstract results of this paper  to the   operator $G.$ Indeed,  Theorems \ref{T3.7}-\ref{T3.9} describe the  spectral properties of the operator $G.$
\bigskip

\subsection*{Acknowledgments}
  Gratitude is expressed to Alexander L. Skubachevskii for   valuable remarks and comments
made during the report devoted to the results of this work and took place 31.10.2017 at the
seminar on differential and functional-differential equations, Department of applied mathematics,
faculty of physics, mathematics and natural Sciences of Peoples' Friendship University of Russia,
Moscow.
The author warmly thanks Tatiana A. Suslina for   valuable comments made during the report
devoted to the  results of this work and took place 06.12.2017 at the seminar of the Department of
mathematical physics St. Petersburg state University, Saint Petersburg branch of V.A. Steklov
Mathematical Institute of the Russian Academy of science, Russia, Saint Petersburg.
Gratitude is expressed to Andrei A. Shkalikov for the number of valuable recommendations
and remarks made during the  report which took place 07.11.2018    at the conference on Partial Differential Equations
and Applications in Memory of Professor B.Yu. Sternin, RUDN University,
Moscow.

\bibliographystyle{amsplain}

\end{document}